\documentclass[12pt,reqno,oneside]{amsart}

\usepackage[letterpaper, total={6in, 9.1in}]{geometry}
\usepackage{comment}
\usepackage{amsmath}
\usepackage{amsfonts}
\usepackage{amssymb}
\usepackage{amsthm}
\usepackage{esint}
\usepackage{mathtools}
\usepackage{array,makecell,longtable}
\usepackage{commath}
\usepackage{xcolor}
\usepackage{graphicx}
\usepackage[toc,title,page]{appendix}
\usepackage{subcaption}
\usepackage{bbm}
\usepackage{hyperref,cite}
\usepackage{color,soul}
\usepackage{bm}
\usepackage{mathrsfs}

\usepackage{tikz}
\usepackage{enumitem}

\newcommand{\Wspace}{\mathcal{P}(\mathbb{R}^n)} 
\newcommand{\T}{\mathbb{T}}
\newcommand{\Pcal}{\mathcal{P}}
\newcommand{\Kcal}{\mathcal{K}}
\newcommand{\Lcal}{\mathcal{L}}
\newcommand{\Fcal}{\mathcal{F}}
\newcommand{\Scal}{\mathcal{S}}
\newcommand{\diff}{\mathrm{d}}

\newtheorem{remark}{Remark}[section]
\newtheorem{theorem}[remark]{Theorem}
\newtheorem{corollary}[remark]{Corollary}
\newtheorem{lemma}[remark]{Lemma}

\newtheorem{definition}[remark]{Definition}

\numberwithin{equation}{section}
\allowdisplaybreaks

\title[Inverse Problems for Transport PDEs on Wasserstein Space]{Inverse Problems for Infinite-dimensional Transport PDEs on Wasserstein Space}

\author{Hongyu Liu}
\address{Department of Mathematics, City University of Hong Kong, Hong Kong SAR, China}
\email{hongyliu@cityu.edu.hk, hongyu.liuip@gmail.com}

\author{Jianliang Qian}
\address{Department of Mathematics, Michigan State University, East Lansing, MI 48824, USA}
\email{jqian@msu.edu}

\author{Shen Zhang}
\address{Department of Mathematics, Michigan State University, East Lansing, MI 48824, USA}
\email{zhan2378@msu.edu}

\begin{document}

	\maketitle
	
	\begin{abstract}
We develop a foundational framework for inverse problems governed by evolutionary partial differential equations (PDEs) on the Wasserstein space of probability measures. While the forward problems for such transport-type PDEs have been extensively and intensively studied, their corresponding inverse problems—which aim to reconstruct unknown operators, cost functions, or interaction kernels from observed solution data—remain largely unexplored at this level of generality.

The cornerstone of our theory is a systematic approach featuring high-order calculus on the Wasserstein space and a progressive variational scheme. This methodology is specifically designed to address the challenges inherent in inverse problems for infinite-dimensional, nonlinear, and nonlocal transport PDEs.

We demonstrate the power and versatility of our theory through two canonical examples: inverse problems for both the Mean Field Control (MFC) Dynamic Programming Equation and the Mean Field Game (MFG) Master Equation. Our work provides, for the first time, a unified foundation for identifying cost functions and interaction kernels from value function data. This establishes a new and fertile field of mathematical research with significant implications for both theory and applications in stochastic control and mean field games.

		\medskip
		

		
	\end{abstract}

\section{Introduction}\label{sect:intro}

The past two decades have witnessed a profound transformation in the analysis of complex, high-dimensional dynamical systems, driven by a paradigm shift towards a geometric viewpoint on the space of probability measures. The ascendance of optimal transport theory, particularly its modern incarnation via the Wasserstein space $\Wspace$, has provided a powerful language to recast evolutionary phenomena—from gradient flows in statistical physics to interacting particle systems—into a unified geometric framework \cite{ambrosio2008,Lions}. This perspective has revealed that many fundamental equations of mathematical physics and applied mathematics are, in essence, {transport partial differential equations (PDEs)} on $\Wspace$, governed by a differential calculus that mirrors and generalizes its Euclidean counterpart.

Among the most significant breakthroughs enabled by this viewpoint is the theory of {Mean Field Games (MFGs)}, introduced independently by Lasry and Lions \cite{lasry2007mean,lasry2006jeux,lasry2006jeux-b,Lions} and Huang, Caines, and Malhamé \cite{huang2006large,Huang2007large}. MFGs model the strategic interaction of a vast population of rational agents, whose individual optimal control problems (described by Hamilton-Jacobi-Bellman equations) are coupled through a shared aggregate population distribution (whose evolution is described by a Fokker-Planck equation). The pinnacle of this theory is the \emph{Master Equation} \cite{cardaliaguet2015master}, a single, infinite-dimensional PDE on $[0,T] \times \mathbb{R}^n \times \Wspace$ that encodes the value function of a generic agent and fully characterizes the game's equilibrium. Similarly, the classical \emph{Dynamic Programming Principle} for a single stochastic controller leads to a Hamilton-Jacobi-Bellman equation, which can itself be interpreted as a PDE on a finite-dimensional projection of the Wasserstein space.

The analysis of these {forward} problems—establishing well-posedness, regularity, and deriving explicit solutions—represents a vibrant frontier in modern PDE theory\cite{mou2025mean,mou2019wellposedness,gangbo2022global,gangbo2022mean}. Indeed, the extension of PDE analysis to infinite-dimensional spaces constitutes a profound evolution in modern mathematics, marking a significant departure from classical finite-dimensional theory. This transition introduces fundamental challenges that demand novel mathematical frameworks: the absence of a canonical Lebesgue measure necessitates alternative notions of integration and differentiation, the choice of topology becomes crucial in defining appropriate solution concepts, and the development of specialized differential calculi becomes {imperative  \cite{Cannarsa2004,ambrosio2008, Lions}.}

Despite these inherent difficulties, the field has witnessed remarkable progress in recent years. Researchers have developed innovative analytical techniques specifically designed to handle the intrinsic nonlinearity, nonlocality, and infinite-dimensional character of these equations. This has led to breakthroughs in understanding the qualitative behavior of solutions and establishing comprehensive existence theories. For comprehensive treatments of the historical development of infinite-dimensional PDEs, we refer to \cite{Volterra1887,Frechet1906,Gateaux1913,Poincare1890,Hille1948,Yosida1948,Kolmogorov1931}, while contemporary advances and state-of-the-art methodologies are detailed in {\cite{hairer2009introduction,otto2021variational,figalli2007existence, Linear_in_Tn, mou2019wellposedness}.}

However, while the analysis of these forward problems has advanced considerably, their {inverse problems} remain a largely uncharted territory of profound theoretical challenge and practical importance. The central question is as fundamental as it is difficult: Given observational data of a system's solution—be it the value function of agents or the temporal evolution of the population distribution—can one reconstruct the underlying governing structures? These structures include the running cost in a control problem, the interaction kernel between agents in an MFG, or the terminal cost shaping their long-term objectives. Solving such inverse problems is the key to \textit{model discovery} and \textit{validation} across the sciences, from inferring rational objectives in economics to calibrating interaction potentials in particle physics.

The existing literature on inverse problems has yet to confront the profound mathematical depth of this setting. Previous work has largely been confined to inverse problems for \textit{finite-dimensional} or \textit{local} PDEs, or specific, often linearized, cases of MFG systems \cite{klibanov2023holder,klibanov2023lipschitz,imanuvilov2023unique,liu2023inverse,Cauchydata-zs,ren2025unique}. There is also a growing body of numerical studies which leverage finite-dimensional approximations for inverse MFGs or optimal transport \cite{chow2022numerical,ding2022mean,guo2024decoding,stuart2020inverse,yu2025equilibrium}. However, a general theory for inverse problems on $\Wspace$ is absent. The core impediments are intrinsic to the domain: the infinite-dimensional, non-linear, and non-local nature of the governing PDEs, combined with the subtle geometric structure of the Wasserstein space itself,  renders classical methodologies—such as the Dirichlet-to-Neumann map or boundary control—largely inapplicable. Consequently, a fundamentally new mathematical architecture is required.

This paper bridges this foundational gap by developing a comprehensive theoretical framework for inverse problems governed by evolutionary transport PDEs on $\Wspace$: 

\begin{enumerate}
    \item We introduce a novel mathematical framework for studying inverse problems associated with infinite-dimensional transport PDEs on $\Wspace$, providing a unified approach to this emerging field.\smallskip
    
    \item We develop a systematic formalism of \emph{high-order calculus on the Wasserstein space} which in combination with a novel \emph{progressive variational scheme} can effectively tackle the intrinsic challenges for these complex inverse problems associated with infinite-dimensional PDEs.\smallskip
    
    \item We establish the first general well-posedness theory (in the sense of Hadamard) for this class of inverse problems, providing sufficient conditions for the global uniqueness of recovering multiple unknowns. \smallskip
    
    \item  We demonstrate that our framework provides a \emph{unified} mathematical foundation, showing that inverse problems for both the \emph{Dynamic Programming Principle} (manifested through Hamilton-Jacobi-Bellman equations) and the \emph{MFG Master Equation} emerge as specific, canonical instances of our general theory. This enables identification of cost functions in optimal control and reconstruction of interaction kernels and terminal costs in mean field games.
\end{enumerate}

By moving from specific instances to a general principle, this work does not only solve two inverse problems but also opens up an entirely new field of research at the intersection of optimal transport, PDE theory, and inverse problems. It provides the mathematical tools to \textit{look backward} into the governing laws of complex systems by observing their collective behavior, a capability with far-reaching implications for science and engineering.

The rest of this paper is organized as follows. Section 2 presents the general framework for inverse problems and states the main global recovery theorems. Section 3 establishes the high-order calculus on the Wasserstein space. Section 4 is dedicated to the application of our theory to the inverse problems for the Dynamic Programming Principle. Section 5 applies the framework to the inverse problems for the MFG Master Equation in $\mathcal{P}(\mathbb{T}^n)$. In Section 6, we consider some inverse boundary value problems for the Master Equation in more general domains.

\section{Statement of main results and technical developments}

\subsection{Statement of main results}
We begin by introducing a general class of infinite-dimensional transport PDEs on the Wasserstein space—specifically, on the space $\mathcal{P}(\mathbb{T}^n)$ of Borel probability measures on the $n$-dimensional torus, $n\in\mathbb{N}$. Endowed with the Wasserstein metric, which quantifies the optimal cost of transporting mass, this space possesses a rich, non-linear geometric structure. This framework fundamentally differentiates it from flat spaces and is the source of the non-local character and intrinsic challenges of these PDEs. Our focus on $\mathbb{T}^n$ is for expository clarity; the extension to a bounded domain $\Omega\subset\mathbb{R}^n$, while more technically involved, follows analogous arguments, and we will treat such inverse problems on $\mathcal{P}(\Omega)$ in Section~6, which inevitably involve boundary conditions. 

Let $U(t, x, m)$, for $(t, x, m)\in \mathbb{R}_+\times\mathbb{T}^n\times\mathcal{P}(\mathbb{T}^n)$, be a real-valued function representing the value function. Consider the following transport PDE, 
\begin{equation}\label{eq:ts1}
\begin{cases}
\displaystyle{ -\partial_t U(t, x, m)+S\left(t, x, D_x^j U, \frac{\delta U}{\delta m},m\right)=0,}\quad\ \\ 
U(T, x, m)=G(x, m),
\end{cases}
\end{equation}
where $G:\mathbb{T}^n\times\mathcal{P}(\mathbb{T}^n)\to\mathbb{R}$ is the terminal cost, $T\in\mathbb{R}_+$ is the terminal time, and $S$ is a real-valued operator. Here, $\frac{\delta U}{\delta m}$ denotes the Lions derivative, and $D_x^j U$, $j\in\mathbb{N}_0$, denotes the spatial derivatives up to order $j\leq 2$. The operator $S$ is critical for the transport process and takes specific forms in different contexts, which will be detailed later.

We primarily address the inverse problem of recovering the operator $S$ from knowledge of the initial data of the value function $U(0, x, m)$ for all $(x, m)\in \mathbb{T}^n\times\mathcal{P}(\mathbb{T}^n)$. This problem can be formulated as the operator equation, 
\begin{equation}\label{eq:ip1}
\mathcal{M}_{G}(S)=U(0, x, m),\quad (x, m)\in \mathbb{T}^n\times\mathcal{P}(\mathbb{T}^n),
\end{equation}
where the measurement operator $\mathcal{M}_G$ maps the unknown $S$ to the data $U(0, x, m)$ via the transport PDE \eqref{eq:ts1} for a given terminal cost $G$.

Our main concern is the unique identifiability issue for the inverse problem \eqref{eq:ip1}:
\begin{equation}\label{eq:u1}
S_1=S_2\quad\mbox{if and only if}\quad U_1(0, x, m)=U_2(0, x, m),\quad (x, m)\in\mathbb{T}^n\times\mathcal{P}(\mathbb{T}^n).
\end{equation}
That is, we aim to establish sufficient conditions that guarantee the bijectivity of the operator $\mathcal{M}_G$ in \eqref{eq:ip1} under general scenarios.

It is crucial to observe that if the measure variable $m$ were replaced by a finite-dimensional parameter, establishing the global uniqueness result in \eqref{eq:u1} would generally be impossible. This stems from the fact that the cardinality—that is, the number of independent variables—of the unknown operator $S$ would exceed that of the available data $U(0, x, m)$. This scenario reduces to the classical, severely ill-posed backward problem for parabolic PDEs. In contrast, we demonstrate that when the transport PDE is formulated on the infinite-dimensional Wasserstein space, unique identifiability can be achieved under general conditions. We corroborate this finding through the analysis of inverse problems for the Dynamic Programming Equation (DPE) in Mean Field Control (MFC) and the Master Equation for Mean Field Games (MFGs), as we will show, though the principle likely extends to other contexts. This establishes a fundamentally new perspective in the theory of inverse problems for PDEs.

\subsection{Technical developments and novelties}

The resolution of inverse problems on the Wasserstein space demands innovative mathematical tools that transcend classical methodologies. Our approach synthesizes several interconnected technical developments, creating a unified framework to address the unique challenges posed by infinite-dimensional, nonlinear, and nonlocal structures.

A cornerstone of our framework is a systematic development of {high-order calculus} on the Wasserstein space, extending substantially beyond the first-order Lions derivative. We establish rigorous definitions for higher-order functional derivatives $\frac{\delta^k U}{\delta m^k}$ via directional variations, complemented by a Taylor expansion theorem for smooth functions. Furthermore, we introduce the concept of analytic functions on $\mathcal{P}(\mathbb{T}^n)$---a notion extendable to general domains---which enables powerful series representations analogous to their classical counterparts. This sophisticated functional-analytic foundation is essential for our progressive variational approach and permits effective handling of the inherent infinite-dimensional nonlinearity.

The methodological core of our work is a {systematic progressive variational  procedure}. Beginning with first-order variation, we derive systems that relate initial distributional variations to solution perturbations. By progressing to higher-order variation through mixed directional derivatives, we extract increasingly detailed information about the Taylor coefficients of the unknown operators. This process is supported by novel extensions of the chain rule to compositions involving Hamiltonian operators. Crucially, this scheme transforms the original nonlinear inverse problems into a hierarchy of linear problems, each revealing successive layers of structural information.

A major technical contribution lies in establishing a {comprehensive well-posedness theory} for the resulting linearized systems. This encompasses the rigorous treatment of Fokker-Planck type equations in distributional spaces, the analysis of coupled MFG linearized systems, and the derivation of uniform estimates and stability results. These developments ensure the mathematical soundness of our variational  procedure and guarantee that the derived identities provide meaningful information for operator recovery.

Our recovery strategy employs multiple innovative identification techniques. We develop a {spectral method} that utilizes eigenfunction expansions of elliptic operators to extract Fourier coefficients of unknown derivatives, and a {static solution analysis} that leverages stationary solutions and quasi-static perturbations to isolate parameter dependencies. For boundary value problems, we extend our approach through a careful utilization of Neumann boundary conditions to extract additional informational content.

The assumption of {analyticity} provides a foundational pillar for our theoretical framework. This property ensures that functions are completely determined by their Taylor expansions, facilitates the construction of explicit examples via convolution representations, and, most importantly, guarantees that recovering all Taylor coefficients through progressive variation implies the complete recovery of the unknown operator. This analytic structure thus provides the precise mechanism through which partial information is amplified to yield full identification.

While our framework establishes a comprehensive foundation for this new class of inverse problems, several aspects invite further investigation. The analyticity assumption, though natural in many applications, could potentially be relaxed through alternative regularity hypotheses. Similarly, the technical conditions in our assumptions, while sufficient for the current results, may not be necessary in all scenarios. Natural extensions to more general cost structures and interaction kernels present promising directions for future work, as does the computational implementation of our progressive variational scheme for practical applications.

Collectively, these developments establish a new paradigm for inverse problems in infinite-dimensional spaces. The potential applications extend well beyond mean field games to diverse domains governed by measure-valued PDEs, including collective behavior models, opinion dynamics, and biological swarm systems. The mathematical architecture we develop opens new avenues for discovering the governing laws of complex systems through observational data of their emergent behavior.

\section{High-Order Calculus on the Wasserstein Space}\label{setup}
    
\subsection{Derivatives of functions on $\mathcal{P}(\mathbb{T}^n)$}

This section develops the high-order calculus theory on the Wasserstein space, which forms a critical foundation for our subsequent study of inverse problems. We note that smooth variations on the Wasserstein space have been studied from various perspectives \cite{bayraktar2023smooth, Linear_in_Tn}.

\begin{definition}\label{def_der_1}
    A function $U : \mathcal{P}(\mathbb{T}^n) \to \mathbb{R}$ is of class $C^1$ if there exists a continuous map $K: \mathcal{P}(\mathbb{T}^n) \times \mathbb{T}^n \to \mathbb{R}$ such that for all $m_1, m_2 \in \mathcal{P}(\mathbb{T}^n)$,
    \begin{equation}\label{derivation}
        \lim_{s \to 0^+} \frac{U(m_1 + s(m_2 - m_1)) - U(m_1)}{s} = \int_{\mathbb{T}^n} K(m_1, x)\,  d(m_2 - m_1)(x).
    \end{equation}
\end{definition}

The mapping $K$ in Definition \ref{def_der_1} is defined up to additive constants. We define the derivative $\frac{\delta U}{\delta m}$ as the unique representative satisfying \eqref{derivation} and the normalization condition, 
\begin{equation}\label{norm_tobe_0}
    \int_{\mathbb{T}^n} \frac{\delta U}{\delta m}(m, x)\,  dm(x) = 0.
\end{equation}

To study the regularity of $\frac{\delta U}{\delta m}$, we recall the Wasserstein distance on $\mathcal{P}(\mathbb{T}^n)$. 

\begin{definition}\label{W_distance}
    For $m_1, m_2 \in \mathcal{P}(\mathbb{T}^n)$, define
    \begin{equation}
        \mathfrak{D}(m_1, m_2) \equiv \sup_{\mathrm{Lip}(\psi) \leq 1} \int_{\mathbb{T}^n} \psi(x)\,  d(m_1 - m_2)(x),
    \end{equation}
    where the Lipschitz constant is given by
    \begin{equation}\label{eq:Lip1}
        \mathrm{Lip}(\psi) = \sup_{\substack{x, y \in \mathbb{T}^n \\ x \neq y}} \frac{|\psi(x) - \psi(y)|}{|x - y|}.
    \end{equation}
\end{definition}

\begin{definition}
    If $U$ is of class $C^1$ and $\frac{\delta U}{\delta m}$ is $C^1$ in the spatial variable, we define the intrinsic derivative $D_m U: \mathcal{P}(\mathbb{T}^n) \times \mathbb{T}^n \to \mathbb{R}^n$ as
    \begin{equation}\label{intrinsic_derivative}
        D_m U(m, x) \equiv D_x \frac{\delta U}{\delta m}(m, x).
    \end{equation}
\end{definition}

\begin{remark} 
    By definition of $\frac{\delta U}{\delta m}$, it follows that \cite{Linear_in_Tn}
    \begin{equation}
        U(m') - U(m) = \int_0^1 \int_{\mathbb{T}^n} \frac{\delta U}{\delta m}((1 - s)m + s m', y)\,  d(m' - m)(y)  ds
    \end{equation}
    for all $m, m' \in \mathcal{P}(\mathbb{T}^n)$.
\end{remark}

A function $U$ is of class $C^2$ if for each $y \in \mathbb{T}^n$, the map $m \mapsto \frac{\delta U}{\delta m}(m, y)$ is $C^1$, and we denote its derivative by $\frac{\delta^2 U}{\delta m^2}(m, y, y')$. Higher-order derivatives $\frac{\delta^k U}{\delta m^k}$ are defined recursively. If $\frac{\delta^k U}{\delta m^k}$ exists and is continuous for all $k \in \mathbb{N}$, then we say $U$ is smooth.

\begin{theorem}\label{taylor}
    If $U: \mathcal{P}(\mathbb{T}^n) \to \mathbb{R}$ is smooth, then for any $m, m' \in \mathcal{P}(\mathbb{T}^n)$ and $N \in \mathbb{N}$,
    \begin{equation}\label{taylor-equation}
        \begin{aligned}
            & U(m') - U(m)\\
             = & \sum_{k=1}^{N-1} \frac{1}{k!} \left( \int_{\mathbb{T}^n} \right)^k \frac{\delta^k U}{\delta m^k}(m, y_0, \dots, y_{k-1}) \prod_{j=0}^{k-1} d(m' - m)(y_j) \\
            & + \frac{1}{(N-1)!} \int_0^1 \left( \int_{\mathbb{T}^n} \right)^N \frac{\delta^N U}{\delta m^N}((1 - s)m + s m', y, y_1, \dots, y_{N-1})\\
            &\prod_{k=0}^{N-1} d(m' - m)(y_k)  ds.
        \end{aligned}
    \end{equation}
\end{theorem}

\begin{proof}
    For a fixed $m' \in \mathcal{P}(\mathbb{T}^n)$, define
    \begin{equation}\label{eq:def u^{1}}
        U^{(1)}(m) = U(m') - U(m) - \int_{\mathbb{T}^n} \frac{\delta U}{\delta m}(m, y)\,  d(m' - m)(y).
    \end{equation}
    Then one has by direct calculations that
    \begin{equation*}
        \begin{aligned}
            & \lim_{s \to 0} \frac{1}{s} [U^{(1)}((1 - s)m + s m'') - U^{(1)}(m)] \\
            = & - \int_{\mathbb{T}^n} \frac{\delta U}{\delta m}(m, y)\,  d(m'' - m)(y) \\
            & - \lim_{s \to 0} \frac{1}{s} \left[ \int_{\mathbb{T}^n} \frac{\delta U}{\delta m}((1 - s)m + s m'', y)\,  d(m' - (1 - s)m - s m'')(y) \right. \\
            & \left. - \int_{\mathbb{T}^n} \frac{\delta U}{\delta m}(m, y)\,  d(m' - m)(y) \right] \\
            = & - \int_{\mathbb{T}^n} \frac{\delta U}{\delta m}(m, y)\,  d(m'' - m)(y) \\
            & - \lim_{s \to 0} \frac{1}{s} \left[ \int_{\mathbb{T}^n} \frac{\delta U}{\delta m}((1 - s)m + s m'', y)\,  d(m' - m)(y) \right. \\
            & \left. - \int_{\mathbb{T}^n} \frac{\delta U}{\delta m}(m, y)\,  d(m' - m)(y) - s \int_{\mathbb{T}^n} \frac{\delta U}{\delta m}((1 - s)m + s m'', y)  d(m - m'')(y) \right] \\
            = & - \lim_{s \to 0} \frac{1}{s} \left[ \int_{\mathbb{T}^n} \frac{\delta U}{\delta m}((1 - s)m + s m'', y)\,  d(m' - m)(y) - \int_{\mathbb{T}^n} \frac{\delta U}{\delta m}(m, y)\,  d(m' - m)(y) \right] \\
            = & - \int_{\mathbb{T}^n} \int_{\mathbb{T}^n} \frac{\delta^2 U}{\delta m^2}(m, y, y')  d(m' - m)(y)\,  d(m'' - m)(y'),
        \end{aligned}
    \end{equation*}
   which implies that 
    \begin{equation}
        \frac{\delta U^{(1)}}{\delta m}(m, y) = - \int_{\mathbb{T}^n} \frac{\delta^2 U}{\delta m^2}(m, y, y')\,  d(m' - m)(y').
    \end{equation}
    By definition \eqref{eq:def u^{1}}, we have $U^{(1)}(m') = 0$. Then
    \begin{equation}
        \begin{aligned}
          &  U^{(1)}(m)  = -[U^{(1)}(m') - U^{(1)}(m)] \\
            & = \int_0^1 \int_{\mathbb{T}^n} \int_{\mathbb{T}^n} \frac{\delta^2 U}{\delta m^2}((1 - s)m + s m', y, y')\,  d(m' - m)(y)  d(m' - m)(y')  ds.
        \end{aligned}
    \end{equation}
    Therefore,
    \begin{equation}
        \begin{aligned}
            & U(m') =  U(m) + \int_{\mathbb{T}^n} \frac{\delta U}{\delta m}(m, y)\,  d(m' - m)(y) \\
            & + \int_0^1 \int_{\mathbb{T}^n} \int_{\mathbb{T}^n} \frac{\delta^2 U}{\delta m^2}((1 - s)m + s m', y, y')\,  d(m' - m)(y)  d(m' - m)(y')  ds.
        \end{aligned}
    \end{equation}
Using notation $U^{(0)}(m)\equiv U(m)$, we define
    \begin{equation}
    \begin{split}
       & U^{(N)}(m) = U(m') - U(m)\\
        & - \sum_{k=1}^{N-1} \frac{1}{k!} \left( \int_{\mathbb{T}^n} \right)^k \frac{\delta^k U}{\delta m^k}(m, y_0, \dots, y_{k-1}) \prod_{j=0}^{k-1}\, d(m' - m)(y_j).
        \end{split}
    \end{equation}
By a similar argument, we have $\eqref{taylor-equation}$ holds.
\end{proof}

\begin{definition}\label{def:analytic}
    A function $U: \mathcal{P}(\mathbb{T}^n) \to \mathbb{R}$ is analytic if it is smooth and
    \begin{equation}
    \begin{split}
        \lim_{N \to \infty} \frac{1}{(N-1)!} \int_0^1 \left( \int_{\mathbb{T}^n} \right)^N &  \frac{\delta^N U}{\delta m^N}((1 - s)m + s m', y, y_1, \dots, y_{N-1})\\
    &      \prod_{k=0}^{N-1} d(m' - m)(y_k)  ds = 0
        \end{split}
    \end{equation}
    for all $m, m' \in \mathcal{P}(\mathbb{T}^n)$.
\end{definition}

\begin{remark}\label{rem:analytic}
We remark that the notion of analyticity on $\mathcal{P}(\mathbb{T}^n)$ is weaker than its classical counterpart on $\mathbb{R}^n$. Indeed, analytic functions on $\mathcal{P}(\mathbb{T}^n)$ can be explicitly constructed from analytic functions on $\mathbb{R}^n$; see Remark~\ref{rmk-analytic} below.
\end{remark}

We now introduce higher-order derivatives via mixed directional derivatives. 

\begin{definition}
    The second-order mixed directional derivative of $U: \mathcal{P}(\mathbb{T}^n) \to \mathbb{R}$ at $m_1$ in directions $(m_2 - m_1, m_3 - m_1)$ is defined as
    \begin{equation}
        \begin{aligned}
            U^{(2)}(m_1) \equiv & \lim_{s_1, s_2 \to 0} \frac{1}{s_1 s_2} [U(m_1 + s_1(m_2 - m_1) + s_2(m_3 - m_1)) \\
            & - U(m_1 + s_1(m_2 - m_1)) - U(m_1 + s_2(m_3 - m_1)) + U(m_1)]
        \end{aligned}
    \end{equation}
    for all $m_1, m_2, m_3 \in \mathcal{P}(\mathbb{T}^n)$.
\end{definition}

\begin{lemma}\label{mixed}
    If $U$ is of class $C^2$, then
    \begin{equation}
        U^{(2)}(m_1; m_2, m_3) = \int_{\mathbb{T}^n} \int_{\mathbb{T}^n} \frac{\delta^2 U}{\delta m^2}(m_1, x, y)  d(m_3 - m_1)(x)  d(m_2 - m_1)(y).
    \end{equation}
\end{lemma}

\begin{proof}
    Since 
    \begin{equation}\label{eq:lemma-mix-1}
        \begin{aligned}
            & \lim_{s_2 \to 0} \frac{U(m_1 + s_1(m_2 - m_1) + s_2(m_3 - m_1)) - U(m_1 + s_1(m_2 - m_1))}{s_2} \\
            = & \int_{\mathbb{T}^n} \frac{\delta U}{\delta m}(m_1 + s_1(m_2 - m_1), x)  d(m_3 - m_1)(x),
        \end{aligned}
    \end{equation}
    and
    \begin{equation}\label{eq:lemma-mix-2}
        \lim_{s_2 \to 0} \frac{-U(m_1 + s_2(m_3 - m_1)) + U(m_1)}{s_2} = -\int_{\mathbb{T}^n} \frac{\delta U}{\delta m}(m_1, x)  d(m_3 - m_1)(x), 
    \end{equation}
    equations \eqref{eq:lemma-mix-1} and \eqref{eq:lemma-mix-2} imply that
    \begin{equation}
        \begin{aligned}
            U^{(2)}(m_1; m_2, m_3) = & \lim_{s_1 \to 0} \frac{1}{s_1} \left[ \int_{\mathbb{T}^n} \frac{\delta U}{\delta m}(m_1 + s_1(m_2 - m_1), x)  d(m_3 - m_1)(x) \right. \\
            & \left. - \int_{\mathbb{T}^n} \frac{\delta U}{\delta m}(m_1, x)  d(m_3 - m_1)(x) \right] \\
            = & \int_{\mathbb{T}^n} \int_{\mathbb{T}^n} \frac{\delta^2 U}{\delta m^2}(m_1, x, y)  d(m_3 - m_1)(x)  d(m_2 - m_1)(y).
        \end{aligned}
    \end{equation}
\end{proof}

\begin{lemma}\label{mixed2}
    Suppose $U: \mathcal{P}(\mathbb{T}^n) \to \mathbb{R}$ is of class $C^k$ ($k \in \mathbb{N}$), and let $h_j = m_j - m$ for $j = 1, \dots, k$. The $k$-th order mixed directional derivative at $m$ is given by 
    \[
        \partial^{k}_{s_1 \cdots s_k} U(m) = \lim_{s_1, \dots, s_k \to 0} \frac{1}{s_1 \cdots s_k} \sum_{\varepsilon \in \{0,1\}^k} (-1)^{k - |\varepsilon|} U\left(m + \sum_{j=1}^k \varepsilon_j s_j h_j \right),
    \]
    where $\varepsilon = (\varepsilon_1, \dots, \varepsilon_k) \in \{0,1\}^k$ and $|\varepsilon| = \varepsilon_1 + \cdots + \varepsilon_k$. Moreover,
    \[
        \partial^{k}_{s_1 \cdots s_k} U(m) = \left( \int_{\mathbb{T}^n} \right)^k \frac{\delta^k U}{\delta m^k}(m, y_1, \dots, y_k)  dh_1(y_1) \cdots dh_k(y_k).
    \]
\end{lemma}

The proof of Lemma \ref{mixed2} follows the reasoning similar to Lemma \ref{mixed} and thus is omitted.
\begin{remark}
    All definitions and results in this section extend naturally to a general bounded domain $\Omega \subset \mathbb{R}^n$. We will apply these results to bounded domains in Section~\ref{bvpmaster} without further justification.
\end{remark}

Finally, we introduce several standard function spaces on $\mathbb{R}^n$ that will be used throughout this work. In the following, we let $\Omega$ denote a bounded domain in $\mathbb{R}^n$ or $\mathbb{T}^n$. 

For $k \in \mathbb{N}$ and $0 < \alpha < 1$, the H\"older space $C^{k+\alpha}(\Omega)$ consists of functions in $C^k(\Omega)$ whose derivatives up to order $k$ are H\"older continuous with exponent $\alpha$. More precisely, for a multi-index $l = (l_1, l_2, \ldots, l_n) \in \mathbb{N}^n$ with $|l| \leq k$, where $D^l \equiv \partial_{x_1}^{l_1} \partial_{x_2}^{l_2} \cdots \partial_{x_n}^{l_n}$, we define the norm,  
\begin{equation*}
    \|\phi\|_{C^{k+\alpha}(\Omega)} \equiv \sum_{|l| \leq k} \|D^l \phi\|_{\infty} + \sum_{|l| = k} \sup_{x \neq y} \frac{|D^l \phi(x) - D^l \phi(y)|}{|x-y|^{\alpha}}.
\end{equation*}

For functions depending on both time and space variables, we define the parabolic H\"older space $C^{\frac{k+\alpha}{2}, k+\alpha}(\mathbb{R} \times \Omega)$ as the set of functions for which $D^l D_t^j \phi$ exist and are H\"older continuous with exponent $\alpha$ in space and $\frac{\alpha}{2}$ in time, for all multi-indices $l \in \mathbb{N}^n$ and $j \in \mathbb{N}$ satisfying $|l| + 2j \leq k$. The corresponding norm is given by 
\begin{equation*}
    \begin{aligned}
        \|\phi\|_{C^{\frac{k+\alpha}{2}, k+\alpha}(\mathbb{R} \times \Omega)} 
        &\equiv \sum_{|l| + 2j \leq k} \|D^l D_t^j \phi\|_{\infty} \\
        &\quad + \sum_{|l| + 2j = k} \sup_{\substack{t \in \mathbb{R} \\ x \neq y}} \frac{|D^l D_t^j \phi(t,x) - D^l D_t^j \phi(t,y)|}{|x-y|^{\alpha}} \\
        &\quad + \sum_{|l| + 2j = k} \sup_{\substack{t \neq t' \\ x \in \Omega}} \frac{|D^l D_t^j \phi(t,x) - D^l D_t^j \phi(t',x)|}{|t-t'|^{\alpha/2}}.
    \end{aligned}
\end{equation*}

Finally, we define the dual space $C^{-(k+\alpha)}(\Omega)$ as the continuous dual of $C^{k+\alpha}(\Omega)$, equipped with the norm,  
\begin{equation*}
    \|f\|_{C^{-(k+\alpha)}(\Omega)} \equiv \sup_{\substack{\phi \in C^{k+\alpha}(\Omega) \\ \|\phi\|_{C^{k+\alpha}(\Omega)} \leq 1}} \langle f, \phi \rangle.
\end{equation*}

\section{Inverse Problems for the Dynamic Programming Equation in Mean Field Control}
\label{ipmfcdpe}
In this section, we address inverse problems associated with the Dynamic Programming Equation (DPE) arising in Mean Field Control (MFC). The DPE constitutes an Eikonal-type equation formulated on the Wasserstein space, modeling a central planner who optimizes a collective cost for a large population of identical agents.

Consider a cost functional of the form, 
\[
J(\alpha) \equiv \int_{t}^{T} \mathbb{E}\left[ \Fcal(u, \Scal(X_u), \alpha_u(X_u)) \right] \diff u + G(\Scal(X_T)),
\]
where $\Scal(X_u) \in \Pcal(\T^n)$ denotes the law of the state process $X_u$, and $\Fcal, G$ are given functions. The value function $U(t, m)$ is defined as
\[
U(t, m) \equiv \inf_{\alpha} J(\alpha) \quad \text{subject to} \quad \Scal(X_t) = m.
\]
This value function satisfies the dynamic programming equation (DPE), 
\begin{equation}
\begin{cases}
-\partial_t U(t, m) + \Lcal\left( \dfrac{\delta U}{\delta m}, m \right) = F(t, m), \\
U(T, m) = G(m),
\end{cases}
\label{eq:mfc}
\end{equation}
where the operator $\Lcal: \Kcal \times \Pcal(\T^n) \to \mathbb{R}$ encodes the Hamiltonian structure, and the space $\Kcal$ is defined as
\[
\Kcal\equiv\left\{K: \Pcal(\T^n)\times \T^n\to\mathbb{R} \mid K \text{ is continuous and satisfies } \eqref{norm_tobe_0}\right\}.
\]
Here, $F$ represents the running cost, while the specific forms of $\Lcal$ and $F$ depend on the structure of $\Fcal$.

The derivation of this DPE follows from the dynamic programming principle applied to McKean-Vlasov optimal control problems. When the state process follows the dynamics, 
\[
\diff X_u = b(X_u, \Scal(X_u), \alpha_u) \diff u + \sigma(X_u, \Scal(X_u), \alpha_u) \diff W_u,
\]
with $b$ and $\sigma$ denoting the drift and diffusion coefficients respectively, 
$W_u$ being the standard Brownian motion, and $\alpha$ being the control process, the value function $U(t, m)$ satisfies the dynamic programming principle. Applying Itô's formula to the value function and taking its linear derivative (Lions derivative) with respect to the measure argument yields the DPE \eqref{eq:mfc}, where the Hamiltonian operator $\Lcal$ takes the form, 
\[
\Lcal\left( \frac{\delta U}{\delta m}, m \right) = \inf_{\alpha} \left\{ \int_{\T^n} \left[ \Fcal(x, m, \alpha(x)) + \mathcal{M}^{\alpha, m}\left[\frac{\delta U}{\delta m}\right](x) \right] \, dm(x) \right\},
\]
with $\mathcal{M}^{\alpha, m}$ being a second-order differential operator corresponding to the drift and diffusion terms of the dynamics.

Since it generalizes the classical Hamilton-Jacobi-Bellman framework to 
measure-valued dynamics, this equation captures macroscopic interactions in collective behavior. We refer the reader to \cite{carmona2018probabilistic,soner2024viscosity} for more details. It falls within the general formulation of \eqref{eq:ts1} with $S = \Lcal - F$.

The specific form of the operator $\Lcal$ depends on the structure of the cost function $\Fcal$ and the dynamics. In this work, we consider the case where $\Lcal$ takes the following specific form, 
\begin{equation}
    \Lcal\left(\frac{\delta U}{\delta m}, m\right) = -\int_{\T^n} \mathrm{div}(D_m U)(t, y)\, \diff m(y) + \int_{\T^n} \nabla f(y) \cdot D_m U(t, y)\, d m(y).
    \label{eq:operator_L}
\end{equation}
This choice corresponds to systems with linear drift term $\nabla f$ and constant diffusion (e.g., Brownian motion), where the cost function $\Fcal$ exhibits separable or convex structure (such as quadratic running costs) and $f$ is a potential function. In this case, the Hamiltonian simplifies to this form, where the first term arises from the Laplacian contribution of the diffusion, and the second term comes from the inner product between the drift and the derivative of the value function. This structure commonly appears in mean field control problems, particularly when the system is driven by gradient flows or when the cost function has linear dependencies (see Assumption 3.1 in \cite{soner2024viscosity}). We again refer to \cite{carmona2018probabilistic,soner2024viscosity} for detailed discussions.

 Throughout this section, we take $f \in C^{\infty}(\T^n)$ and focus on the case where the operator $\Lcal$ takes the form \eqref{eq:operator_L}. Indeed, our methodology extends naturally to scenarios with reduced regularity requirements on $f$ and more general forms of the operator $\mathcal{L}$. We present this specific example to demonstrate our approach with clarity and transparency.

Moreover, we require that \(F\) and \(G\) in \eqref{eq:mfc} satisfy the following regularity conditions. 

\begin{enumerate}
    \item \textbf{First-order regularity:} For \(\alpha \in (0,1)\),
        \begin{equation}
            \sup_{\substack{m \in \mathcal{P}(\mathbb{T}^n) \\ t \in [0,T]}} 
            \left( \|F(t, m)\| + 
            \left\| \frac{\delta F}{\delta m}(t, m, \cdot) \right\|_{2+\alpha} \right) 
            + \mathrm{Lip} \left( \frac{\delta F}{\delta m} \right) < \infty,
        \end{equation}
        where the Lipschitz constant is defined by
        \begin{equation}\label{eq:Lip_F}
            \mathrm{Lip} \left( \frac{\delta F}{\delta m} \right)(t) 
            \equiv \sup_{m_1 \neq m_2} [\mathfrak{D}(m_1, m_2)]^{-1} 
            \left\| \frac{\delta F}{\delta m}(t, m_1, \cdot) 
            - \frac{\delta F}{\delta m}(t, m_2, \cdot) \right\|_{1+\alpha}
        \end{equation}
        for all \(m_1, m_2 \in \mathcal{P}(\mathbb{T}^n)\).  
        Similarly, for \(G\),
        \begin{equation}
            \sup_{m \in \mathcal{P}(\mathbb{T}^n)} 
            \left( \|G(\cdot, m)\|_{\alpha} 
            + \left\| \frac{\delta G}{\delta m}(\cdot, m, \cdot) \right\|_{\alpha, 2+\alpha} \right)
            + \mathrm{Lip} \left( \frac{\delta G}{\delta m} \right) < \infty,
        \end{equation}
        where
        \begin{equation}\label{eq:Lip_G}
            \mathrm{Lip} \left( \frac{\delta G}{\delta m} \right)
            \equiv \sup_{m_1 \neq m_2} [\mathfrak{D}(m_1, m_2)]^{-1}
            \left\| \frac{\delta G}{\delta m}(\cdot, m_1, \cdot) 
            - \frac{\delta G}{\delta m}(\cdot, m_2, \cdot) \right\|_{\alpha, 1+\alpha}
        \end{equation}
        for all \(m_1, m_2 \in \mathcal{P}(\mathbb{T}^n)\).

    \item \textbf{Higher-order regularity:} For all \(k \in \mathbb{N}\),
        \begin{equation}
            \sup_{\substack{m \in \mathcal{P}(\mathbb{T}^n) \\ t \in [0,T]}} 
            \left\| \frac{\delta^k F}{\delta m^k}(t, m, y, y_1, \dots, y_{k-1}) \right\|_{(2+\alpha)^k} +\mathrm{Lip}\left(\frac{\delta^k F}{\delta m^k}\right)
            < \infty,
        \end{equation}
        and
        \begin{equation}
            \sup_{m \in \mathcal{P}(\mathbb{T}^n)} 
            \left\| \frac{\delta^k G}{\delta m^k}(m, y, y_1, \dots, y_{k-1}) \right\|_{(2+\alpha)^k} + +\mathrm{Lip}\left(\frac{\delta^k G}{\delta m^k}\right)
            < \infty.
        \end{equation}
\end{enumerate}

We remark in passing that the first condition concerns standard regularity requirements for \(F\), while the second imposes higher-order regularity conditions essential for our analysis.

\subsection{Well-posedness results for the forward DPE}

We begin with the following definition of a classical solution. 

\begin{definition}\label{definition-of-solution-1}
    A map $U: (0,T) \times \mathcal{P}(\mathbb{T}^n) \to \mathbb{R}$ is a \emph{classical solution} of equation \eqref{eq:mfc} if it satisfies that 
    \begin{enumerate}
        \item $U$ is continuous in all arguments (with respect to the $ \mathfrak{D}$ distance), and for each $m \in \mathcal{P}(\mathbb{T}^n)$, the mapping $t \mapsto U(t, m)$ belongs to $C^1((0,T) \times \mathcal{P}(\mathbb{T}^n))$;
        
        \item $U$ is of class $C^1$ with respect to the measure variable $m$, and its first-order derivative $\dfrac{\delta U}{\delta m}(t, m, y)$ is continuous in $t$ and $y$;
        
        \item $U$ satisfies equation \eqref{eq:mfc} pointwise.
    \end{enumerate}
\end{definition}

Associated with the DPE \eqref{eq:mfc}, we introduce the following auxiliary system for $u(t)$ and $m(t, x)$, 
\begin{equation}\label{corr.sys}
    \begin{cases}
        -\dfrac{d}{dt} u(t) = F(t, m(t,\cdot)), & \text{in } (t_0, T), \\
        \partial_t m - \Delta m - \nabla \cdot (m \nabla f) = 0, & \text{in } (t_0, T) \times \mathbb{T}^n, \\
        m(t_0, x) = m_0 (x), & \text{in } \mathbb{T}^n, \\
        u(T) = G(m_T), & \text{in } \mathbb{T}^n.
    \end{cases}  
\end{equation}

In system \eqref{corr.sys}, $m$ and $m_0$ are interpreted as density functions. For the initial data $m_0 \in \mathcal{P}(\mathbb{T}^n) \cap C^{\alpha}(\mathbb{T}^n)$, the well-posedness of \eqref{corr.sys} follows directly from the classical theory for heat-type equations\cite{evans2010pde}.

We note that while $F(t, m)$ in equation \eqref{eq:mfc} is defined on $[0,T] \times \mathcal{P}(\mathbb{T}^n)$, we extend its domain in \eqref{corr.sys} to include functions
\begin{equation*}
    \left\{ m(t,x): (0,T) \times \mathbb{T}^n \to \mathbb{R} \,\middle|\, \int_{\mathbb{T}^n} m(t,x)\, dx = 1 \text{ for all } t \in (0,T) \right\}.
\end{equation*}
This extension is well-defined since $m_0 \in \mathcal{P}(\mathbb{T}^n) \cap C^{\alpha}(\mathbb{T}^n)$.

The well-posedness of equation \eqref{eq:mfc} is established by the following result.

\begin{theorem}\label{one-to-one}
    Let $u(t)$ be the solution to system \eqref{corr.sys} with initial data $m_0\in \mathcal{P}(\mathbb{T}^n) \cap C^{\alpha}(\mathbb{T}^n)$. Define a function $U(t, m): (0,T) \times \mathcal{P}(\mathbb{T}^n) \to \mathbb{R}$ by $U(t_0, m_0) = u(t_0)$. Then $U(t, m)$ is the unique classical solution to equation \eqref{eq:mfc}.
\end{theorem}

Prior to proving Theorem~\ref{one-to-one}, we state the following regularity result. 

\begin{theorem}\label{diff_wrt_m}
    Let $m_0 \in \mathcal{P}(\mathbb{T}^n)$. The mapping $U(t, m)$ defined in Theorem~\ref{one-to-one} is of class $C^1$ with respect to the measure variable. Moreover, the first-order derivative $\dfrac{\delta U}{\delta m}(t, m, y)$ is continuous in $t$ and $y$.
\end{theorem}

The proof of Theorem~\ref{diff_wrt_m} is postponed to Subsection~\ref{differentiability of U}, and we first present the proof of Theorem~\ref{one-to-one}.

\begin{proof}[Proof of Theorem \ref{one-to-one}]
    \textbf{Existence.} We first assume $m_0\in\mathcal{P}(\mathbb{T}^n)\cap C^{\infty}(\mathbb{T}^n)$. Let $(u,m)$ be the solution of $\eqref{corr.sys}$ with initial data $m(t_0,x)=m_0(x)$. We begin by computing the time derivative of $U(t,m)$ at $(t_0,m_0)$, 
    \begin{equation}\label{eq:p U-1}
        \begin{aligned}
           & \partial_t U(t_0,m_0) = \lim_{h \to 0} \frac{U(t_0+h, m_0) - U(t_0, m_0)}{h} \\
            &= \lim_{h \to 0} \frac{U(t_0+h, m_0) - U(t_0+h, m(t_0+h,\cdot)) + U(t_0+h, m(t_0+h,\cdot)) - U(t_0, m_0)}{h}.
        \end{aligned}
    \end{equation}
    
    By definition of $U(t, m)$, the second-difference quotient satisfies that 
    \begin{equation}\label{eq:p U-2}
    \begin{aligned}
         \lim_{h \to 0} \frac{U(t_0+h, m(t_0+h,\cdot)) - U(t_0, m_0)}{h} &= \lim_{h \to 0} \frac{u(t_0+h) - u(t_0)}{h} \\
         &= (\frac{d}{dt}u)(t_0)=F\left(t_0,m(t_0,\cdot)\right).
    \end{aligned} 
    \end{equation}
    
    For the first-difference quotient, Theorem~\ref{diff_wrt_m} ensures that $U$ is of class $C^1$ with respect to $m$. Applying Definitions~\eqref{derivation} and~\eqref{intrinsic_derivative}, we obtain 
    \begin{equation}\label{eq:p U-3}
        \begin{aligned}
            & \lim_{h \to 0} \frac{U(t_0+h, m_0) - U(t_0+h, m(t_0+h,\cdot))}{h} \\
            =& \lim_{h \to 0} \int_0^1 \int_{\mathbb{T}^n} \frac{\delta U}{\delta m}(t_0+h, (1-s)m(t_0) + s m(t_0+h), y) \left[ \frac{m(t_0+h,y) - m(t_0,y)}{h} \right] dy\, ds \\
            =& \int_0^1 \int_{\mathbb{T}^n} \frac{\delta U}{\delta m}(t_0, m(t_0), y) \partial_t m(t_0,y)\, dy\, ds \\
            =& \int_{\mathbb{T}^n} \frac{\delta U}{\delta m}(t_0, m_0, y) \left[ \Delta m(t_0,y)+ \nabla \cdot (m(t_0,y)\nabla f) \right] dy \\
            =& \int_{\mathbb{T}^n} \mathrm{div}(D_m U)(t_0, y) m(t_0,y)\, dy - \int_{\mathbb{T}^n} \nabla f(y) \cdot D_m U(t_0, y) m(t_0,y)\, dy.
        \end{aligned}
    \end{equation}
    
    Furthermore, by definition of $u(t,x)$, we have $U(T, m_0) = \overline{u}(T, x)$, where $\overline{u}$ satisfies 
    \begin{equation}
        \begin{cases}
            \overline{m}(T, x) = m_0(x),  \, x\in\mathbb{T}^n,\\
            \overline{u}(T) = G(\overline{m}_T),
        \end{cases}  
    \end{equation}
for some function $\overline{m}$. Hence, $U(T, m_0) = \overline{u}(T) = G(\overline{m}_T) = G(m_0)$.
    
    Using the results in ~\eqref{eq:p U-1},~\eqref{eq:p U-2}, and ~\eqref{eq:p U-3}, we conclude that $U(t_0, m_0)$ satisfies 
    equation~\eqref{eq:mfc} for all $t_0\in(0,T)$ and $m_0\in\mathcal{P}(\mathbb{T}^n)\cap C^{\infty}(\mathbb{T}^n)$; by a density argument, it holds for all $m_0\in\mathcal{P}(\mathbb{T}^n)$.
    
    \medskip
    
    \textbf{Uniqueness.} Let $V$ be another solution of equation~\eqref{eq:mfc}. For given $t_0$ and $m_0\in \mathcal{P}(\mathbb{T}^n)$, let $m(t, x)$ solve 
    \begin{equation}
        \begin{cases}
            \partial_t m - \Delta m + \nabla \cdot (m \nabla f) = 0, & \text{in } (t_0, T) \times \mathbb{T}^n, \\
            m(t_0, \cdot) = m_0(\cdot), & \text{in } \mathbb{T}^n.
        \end{cases}
    \end{equation}
    
    Define $v(t) \equiv V(t, m(t, \cdot))$ and $m(t) \equiv m(t, \cdot)$. Then, we have 
    \begin{equation*}
        \begin{aligned}
            \frac{d}{dt} v(t) &= \partial_t V(t, m(t, x)) + \int_{\mathbb{T}^n} \frac{\delta V}{\delta m}(t, m, y) \partial_t m(t,y)\, dy \\
            &= \partial_t V(t, m(t, x)) + \int_{\mathbb{T}^n} \frac{\delta V}{\delta m}(t, m, y) \left[ \Delta m +\nabla \cdot (m \nabla f) \right] dy \\
            &= \partial_t V(t, m(t, x)) - \mathcal{L}\left( \frac{\delta V}{\delta m}, m \right) \\
            &= -F(t, m).
        \end{aligned}
    \end{equation*}
    
    Since $v(T) = V(T, m(T, x)) = G(m(T, x))$, the pair $(v(t), m)$ also solves system~\eqref{corr.sys}. By uniqueness for that system, we conclude that $v(t) = u(t)$, and hence $U(t, m) = V(t, m)$.
\end{proof}

\subsection{Differentiability of $U$ with respect to the measure}\label{differentiability of U}

Let us fix $(t_0, m_0) \in [0,T] \times \mathcal{P}(\mathbb{T}^n)$ and let $(u, m)$ be the solution of the system \eqref{corr.sys} with the initial condition $m(t_0, \cdot) = m_0$. For a given $\mu_0 \in C^{-(1+\alpha)}(\mathbb{T}^n)$, we consider the following linearized system around $(u,m)$, 
\begin{equation}\label{corr.sys.linear}
    \begin{cases}
       \displaystyle{ -\frac{d}{dt} v(t) = \frac{\delta F}{\delta m}(t, m(t))(\mu(t)),} & \text{in } (t_0, T), \smallskip\\
        \partial_t \mu - \Delta \mu - \nabla \cdot (\mu \nabla f) = 0, & \text{in } (t_0, T) \times \mathbb{T}^n, \\
        \mu(t_0, \cdot) = \mu_0(\cdot), & \text{in } \mathbb{T}^n,\smallskip \\
       \displaystyle{ v(T) = \frac{\delta G}{\delta m}(m(T))(\mu(T)), } & \text{in } \mathbb{T}^n.
    \end{cases}
\end{equation}
Here, we use the notation
\begin{equation}
    \frac{\delta F}{\delta m}(t, m(t))(\mu(t)) \equiv \int_{\mathbb{T}^n} \frac{\delta F}{\delta m}(t, m(t), y) \mu(t, y)\,  dy,
\end{equation}
with an analogous definition for $G$.

In this subsection, we aim to show that the derivative $\frac{\delta U}{\delta m}$ exists and satisfies
\begin{equation}\label{key}
    \int_{\mathbb{T}^n} \frac{\delta U}{\delta m}(t, m, y) \mu_0(y)\,  dy = v(t).
\end{equation}

We begin by studying the well-posedness of the Fokker-Planck equation in the distributional sense.

\begin{lemma}\label{FP-eq-wp}
    Let $b \in [L^{\infty}(\mathbb{T}^n)]^n$ and $\mu_0 \in C^{-(1+\alpha)}(\mathbb{T}^n)$. Then there exists a unique distributional solution to
    \begin{equation}\label{FP_equation}
        \begin{cases} 
            \partial_t \mu - \Delta \mu - \nabla \cdot (b\mu) = 0, & \text{in } (t_0, T) \times \mathbb{T}^n, \\
            \mu(t_0, \cdot) = \mu_0(\cdot), & \text{in } \mathbb{T}^n.
        \end{cases}
    \end{equation}
    Moreover, there exists a constant $C > 0$ such that the solution satisfies
    \begin{equation}\label{ess-FP}
        \sup_{t \in [t_0, T]} \|\mu(t)\|_{-(1+\alpha)} \leq C \|\mu_0\|_{-(1+\alpha)}.
    \end{equation}
    Furthermore, if $\mu_0^1, \mu_0^2 \in \mathcal{P}(\mathbb{T}^n)$, then
    \begin{equation}\label{ess-FP2}
        \sup_{t \in [t_0, T]} \mathfrak{D}(\mu_1(t) , \mu_2(t)) \leq C \mathfrak{D}(\mu_0^1 , \mu_0^2),
    \end{equation}
    where $\mu_1$ and $\mu_2$ are solutions of \eqref{FP_equation} with initial conditions $\mu_0^1$ and $\mu_0^2$, respectively.
\end{lemma}

\begin{proof}
    The estimate \eqref{ess-FP} can be directly verified when $\mu_0 \in C^{2+\alpha}(\mathbb{T}^n)$. For the general case, it follows by a standard approximation argument. We refer to \cite{linear_in_Om} for complete details.

    To establish \eqref{ess-FP2}, consider a test function $\phi$ satisfying
    \begin{equation}
        \begin{cases} 
            -\partial_t \phi - \Delta \phi - b \cdot \nabla \phi = 0, & \text{in } (t_0, t) \times \mathbb{T}^n, \\
            \phi(t, \cdot) = \psi (\cdot), & \text{in } \mathbb{T}^n,
        \end{cases}
    \end{equation}
    where $\psi \in W^{1,\infty}$ with $\text{Lip}(\psi) \leq 1$. By duality, we have
    \begin{equation}
        \int_{\mathbb{T}^n} \psi(x)  \mu_i(t, x)  dx = \int_{\mathbb{T}^n} \phi(t_0, x)  \mu_0^i(x)  dx, \quad i = 1, 2.
    \end{equation}
    Subtracting these identities yields
    \begin{equation}
        \begin{aligned}
            \int_{\mathbb{T}^n} \psi(x) & (\mu_1(t, x) - \mu_2(t, x))  dx \\
            &= \int_{\mathbb{T}^n} \phi(t_0, x) (\mu_0^1(x) - \mu_0^2(x))  dx \\
            &\leq C \mathfrak{D}(\mu_0^1 ,\mu_0^2),
        \end{aligned}
    \end{equation}
    where the inequality follows from the regularity of $\phi(t_0, \cdot)$. Taking the supremum over all $\psi \in W^{1,\infty}$ with $\text{Lip}(\psi) \leq 1$ and over $t \in [t_0, T]$, we obtain \eqref{ess-FP2}.
\end{proof}

As an immediate consequence, we have the following corollary. 
\begin{corollary}
    Under the assumptions of Lemma \ref{FP-eq-wp}, the system \eqref{corr.sys.linear} admits a unique solution $(v, \mu)$ satisfying
    \begin{equation}
     \|v(t)\|_{C^1([t_0,T])} + \sup_{t \in [t_0, T]} \|\mu(t)\|_{-(1+\alpha)} \leq C \|\mu_0\|_{-(1+\alpha)}.
    \end{equation}
\end{corollary}

We now prove that the linearized system possesses a fundamental solution, which will be identified with the derivative $\frac{\delta U}{\delta m}$.

\begin{lemma}\label{fun-solution}
    Consider the system \eqref{corr.sys.linear}. There exists a function $S: [0, T] \times \mathcal{P}(\mathbb{T}^n) \times \mathbb{T}^n \to \mathbb{R}$ such that
    \begin{equation}\label{fun-solu-intergral}
        v(t_0) = \int_{\mathbb{T}^n} S(t_0, m_0, y) \mu_0(y)  dy
    \end{equation}
    for all $(t_0, m_0, \mu_0)$. Moreover, there exists a constant $C > 0$ such that
    \begin{equation}
        \sup_{(t, m) \in [0,T] \times \mathcal{P}(\mathbb{T}^n)} \|S(t, m, \cdot)\|_{1+\alpha} \leq C.
    \end{equation}
\end{lemma}

\begin{proof}
    Let $v(t; \mu_0)$ denote the solution of the first equation in \eqref{corr.sys.linear} with the initial condition $\mu(t_0) = \mu_0$. Define
    \begin{equation}
        S(t_0, m_0, y) = v(t_0; \delta_y),
    \end{equation}
    where $\delta_y$ denotes the Dirac measure centered at $y \in \mathbb{T}^n$.

    By the linearity of the system \eqref{corr.sys.linear}, we have
    \begin{equation}
        \frac{\partial S}{\partial y_j}(t_0, m_0, y) = v\left(t_0; -\frac{\partial \delta_y}{\partial y_j}\right),
    \end{equation}
    where the derivative of $\delta_y$ is understood in the distributional sense. Since $\frac{\partial \delta_y}{\partial y_j}$ is bounded in $C^{-(1+\alpha)}$, Lemma \ref{FP-eq-wp} implies that 
    \begin{equation}
        \sup_{(t, m)\in [0,T] \times \mathcal{P}(\mathbb{T}^n)} \|S(t, m, \cdot)\|_{1+\alpha} \leq C.
    \end{equation}

    The representation formula \eqref{fun-solu-intergral} follows directly from the linearity of the system \eqref{corr.sys.linear}.
\end{proof}

Now we are ready to prove Theorem \ref{diff_wrt_m}.

\begin{proof}[Proof of Theorem \ref{diff_wrt_m}]
    Let $m_0^1, m_0^2 \in C^\alpha \cap \mathcal{P}(\mathbb{T}^n)$. Fix $t_0 \in [0,T]$, and let $(u_1, m_1)$ and $(u_2, m_2)$ be solutions of the system \eqref{corr.sys} with initial conditions $m_1(t_0) = m_0^1$ and $m_2(t_0) = m_0^2$, respectively.
    
    Consider the linearized system around $(u_2, m_2)$ with the initial condition $\mu(t_0) = m_0^1 - m_0^2$, 
    \begin{equation}\label{corr.sys.linear2}
        \begin{cases}
          \displaystyle{  -\frac{d}{dt} v(t) = \frac{\delta F}{\delta m}(t, m_2(t))(\mu(t)),} &\text{ in } (t_0,T),\smallskip\\
            \partial_t \mu - \Delta \mu - \nabla \cdot (\mu \nabla f) = 0,  &\text{ in } (t_0,T)\times\mathbb{T}^n,\\
            \mu(t_0, \cdot) = m_0^1 - m_0^2,  &\text{ in } \mathbb{T}^n,\smallskip\\
           \displaystyle{ v(T, \cdot) = \frac{\delta G}{\delta m}(m_2(T))(\mu(T)),}&\text{ in } \mathbb{T}^n.
        \end{cases}  
    \end{equation}

    Let $(v, \mu)$ be the solution to \eqref{corr.sys.linear2}, and define the remainder terms
    \[
    (z, \rho) = (u_1 - u_2 - v, \ m_1 - m_2 - \mu).
    \]
    Then $(z, \rho)$ satisfies the system 
    \begin{equation}\label{eq:(z,rho)}
        \begin{cases}
           \displaystyle{ -\frac{d}{dt} z(t) = F(t, m_1) - F(t, m_2) - \frac{\delta F}{\delta m}(t, m_2(t))(\mu(t)),} &\text{ in } (t_0,T),\smallskip\\
            \partial_t \rho - \Delta \rho - \nabla \cdot (\rho \nabla f) = 0,  &\text{ in } (t_0,T)\times\mathbb{T}^n,\\
            \rho(t_0, \cdot) = 0,  &\text{ in } \mathbb{T}^n,\\
          \displaystyle{  z(T, x) = G(m_1(T)) - G(m_2(T)) - \frac{\delta G}{\delta m}(m_2(T))(\mu(T)),} &\text{ in } \mathbb{T}^n.
        \end{cases}
    \end{equation}

    Observing that $\mu = m_1 - m_2$, we apply Theorem \ref{taylor} to obtain 
    \begin{equation}\label{eq:A(t)}
       \begin{aligned}
     A(t)\equiv&F(t,m_1)-F(t,m_2)-\frac{\delta F}{\delta m}(t,m_2(t))(\mu)\\
   =&F(t,m_1)-F(t,m_2)-\int_{\mathbb{T}^n}\frac{\delta F}{\delta m}(t,m_2(t),y)(m_1-m_2)d(y)\\
   =&\int_0^1\int_{\mathbb{T}^n}\int_{\mathbb{T}^n}\frac{\delta^2 F}{\delta m^2}(  (1-s)m_2+m_1,y,y')(m_1(t,y)-m_2(t,y))\\
   &\hspace*{3cm}  dy(m_1(t,y')-m_2(t,y'))dy'ds.
\end{aligned}
    \end{equation}
    By the regularity assumptions on $F$ and ~\eqref{eq:A(t)}, we estimate 
    \begin{equation}\label{eq:ess-A}
        \begin{aligned}
            \|A(t)\|_\infty &\leq C \int_0^1 \int_{\mathbb{T}^n} \mathfrak{D}(m_1, m_2) \left\| D_{y'} \frac{\delta F}{\delta m}(t, m_2 + s(m_1 - m_2), y, \cdot) \right\|_{L^\infty(\mathbb{T}^n)}\\
            &\hspace*{3cm}   d(m_1 - m_2)(y)  ds \\
            &\leq C [\sup_{t \in [0,T]}\mathfrak{D}(m_1, m_2)]^2.
        \end{aligned}
    \end{equation}
    Similarly, for the terminal condition, we have 
    \begin{equation}\label{eq:ess-G}
        \left\| G(m_1(T)) - G(m_2(T)) - \frac{\delta G}{\delta m}(m_2(T))(\mu(T)) \right\|_\infty \leq C [\sup_{t \in [0,T]}\mathfrak{D}(m_1, m_2)]^2.
    \end{equation}
   Using the estimates ~\eqref{eq:ess-A} and ~\eqref{eq:ess-G} in ~\eqref{eq:(z,rho)}, we obtain the estimate for $z$,  
    \begin{equation}\label{eq:ess-z}
        \|z\|_\infty \leq C [\sup_{t \in [0,T]}\mathfrak{D}(m_1 , m_2)]^2.
    \end{equation}
   By definitions of $m_1(t)$ and $m_2(t)$ and using Lemma~\ref{FP-eq-wp}, we obtain the estimate
\begin{equation}\label{eq:ess:m12}
    \sup_{t \in [0,T]} \mathfrak{D}\bigl(m_1(t), m_2(t)\bigr) \leq C \, \mathfrak{D}(m_0^1, m_0^2).
\end{equation}
    By Lemma \ref{fun-solution} we have 
    \begin{equation}\label{eq:v(t_0)}
        v(t_0) = \int_{\mathbb{T}^n} S(t_0, m_2, y) (m_0^1 - m_0^2)(y)\,  dy.
    \end{equation}
   By definitions of $u_1$, $u_2$, and $\mu$, combining \eqref{eq:ess-z}, ~\ref{eq:ess:m12}, and \eqref{eq:v(t_0)} yields
    \begin{equation}
    \begin{split}
       & \left\| U(t_0, m_0^1) - U(t_0, m_0^2) - \int_{\mathbb{T}^n} S(t_0, m_2, y) (m_0^1 - m_0^2)(y)\,  dy \right\|_\infty\\
        \leq &\ C [\mathfrak{D}(m_0^1 ,m_0^2)]^2.
        \end{split}
    \end{equation}
    This establishes the differentiability of $U$ with respect to $m$, with
    \begin{equation}
        \frac{\delta U}{\delta m}(t, m, y) = S(t, m, y),
    \end{equation}
    up to an additive constant.
    
    Finally, we verify the normalization condition for $S$. Taking $\mu_0 = m_0$ in system \eqref{corr.sys.linear}, we observe that the solution is $(v, \mu) = (0, m)$, since  
    \[
    \frac{\delta F}{\delta m}(t, m(t))(m(t)) = 0 \quad \text{and} \quad \frac{\delta G}{\delta m}(m(T))(m(T)) = 0
    \]
    by definition of the derivative with respect to the measure variable.
 Therefore, by Lemma \ref{fun-solution}, we conclude that $S$ satisfies the normalization condition, 
    \begin{equation}
        \int_{\mathbb{T}^n} S(t, m_0, y) m_0(y)\,  dy = 0,
    \end{equation}
    which completes the proof.
\end{proof}

\begin{corollary}
    If the $k$-th order regularity assumptions for $F$ and $G$ hold, then $U(t,m)$ is $k$-th order differentiable with respect to $m$.
\end{corollary}

\begin{proof}
    Since $\frac{\delta U}{\delta m}(m,y)$ exists, we consider the first-order mixed directional derivative in the direction $\mu_0$. Define $U^1(t,m) \equiv \partial_{s_1} U(t,m;\mu_0)$. Then $U^1(t,m)$ satisfies the linearized system, 
    \begin{equation}
        \begin{cases}
            \displaystyle{-\partial_t U^1(t,m) + \mathcal{L}\left(\frac{\delta U^1}{\delta m}, m\right) = \int_{\mathbb{T}^n} \frac{\delta F}{\delta m}(t,m,y)  \, {d}\mu_0(y),}\smallskip \\
           \displaystyle{ U^1(T,m) = \int_{\mathbb{T}^n} \frac{\delta G}{\delta m}(m,y)\,  {d}\mu_0(y).}
        \end{cases}
    \end{equation}
    
    Employing arguments similar to those in the above, we conclude that $\frac{\delta U^1}{\delta m}(t,m,y';\mu_0)$ exists. By the definition of $U^1$, we have 
    \begin{equation}
        \begin{aligned}
            & \int_{\mathbb{T}^n} \frac{\delta U^1}{\delta m}(t,m,y';\mu_0)\,  {d}(m' - m)(y') \\
            = &  \lim_{s \to 0} \frac{1}{s} \left[ \int_{\mathbb{T}^n} \frac{\delta U}{\delta m}(t, m + s(m' - m), y)  \, {d}\mu_0(y) - \int_{\mathbb{T}^n} \frac{\delta U}{\delta m}(t, m, y)  \, {d}\mu_0(y) \right].
        \end{aligned}
    \end{equation}
    
    Therefore, if $\frac{\delta^2 U}{\delta m^2}(t,m,y,y')$ exists, then it must satisfy that 
    \begin{equation}
        \int_{\mathbb{T}^n} \frac{\delta^2 U}{\delta m^2}(t,m,y,y')  \, {d}\mu_0(y) = \frac{\delta U^1}{\delta m}(t,m,y';\mu_0),
    \end{equation}
    for all directions $\mu_0$.
    
    To construct the second derivative explicitly, for $k \in \mathbb{Z}^n$, $\mathrm{i}\equiv\sqrt{-1}$, we consider the Fourier basis ${d}\mu_{0k}(y) = e^{-2\pi \mathrm{i} k\cdot y}  {d}y$ and define 
    \begin{equation}
        \frac{\delta^2 U}{\delta m^2}(t,m,y,y') \equiv \sum_{k \in \mathbb{Z}} \frac{\delta U^1}{\delta m}(t,m,y';\mu_{0k}) e^{2\pi \mathrm{i} k\cdot y}.
    \end{equation}

    One can verify that this expression is well-defined and indeed represents the second-order derivative of $U(t,m)$.
    
    The result for general $k$-th order differentiability follows by repeating this argument and applying mathematical induction.
\end{proof}

\subsection{Inverse problems for the dynamic programming equation}

We have established that equation \eqref{eq:mfc} admits a unique classical solution in the sense of Definition \ref{definition-of-solution-1}. In this subsection, we address the inverse problem \eqref{eq:ip1} within this specific framework.

Let $\mathcal{U}$ denote the space of all real-valued functions defined on $\mathbb{T}^n$. We define the operator $\mathcal{M}$ mapping running costs to value functions by 
\begin{equation}\label{eq:dpem1}
    \mathcal{M}(F) \equiv U(0, m),
\end{equation}
where $U(t,m)$ is the solution of system \eqref{eq:mfc} with running cost $F(t,m)$. In this context, $\mathcal{M}(F)$ corresponds precisely to the operator $\mathcal{M}_G(S)$ introduced in \eqref{eq:ip1}.

The main result of this section characterizes the injectivity properties of the operator $\mathcal{M}$. Since we need to require that $F(x,m)$ is analytic in this result, we first verify that this analyticity assumption is compatible with our regularity conditions introduced at the beginning of this section on $F$. 

\begin{remark}\label{rmk-analytic}
    The notion of analyticity introduced in Definition~\ref{def:analytic} (cf. Remark~\ref{rem:analytic}) represents a natural generalization of the classical definition. To illustrate this, consider the following construction \cite{Linear_in_Tn}. Let $\Phi(z,\theta): \mathbb{R}^n \times \mathbb{R} \to \mathbb{R}$ and $\rho: \mathbb{T}^n \to \mathbb{R}$ be smooth functions, and define $F$ by 
    \begin{equation}
        F(m) = \int_{\mathbb{R}^n} \Phi(z, (\rho * m)(z))\,  dz,
    \end{equation}
    where the convolution is given by 
    \begin{equation}
        (\rho * m)(z) \equiv \int_{\mathbb{T}^n} \rho(y - z) \, dm(y).
    \end{equation}
    The functional derivatives of $F$ are then 
    \begin{align}
        \frac{\delta F}{\delta m}(m,y) &= \int_{\mathbb{R}^n} \frac{\partial \Phi}{\partial \theta}(z, \rho * m(z)) \rho(y - z)\,  dz, \\
        \frac{\delta^k F}{\delta m^k}(m,y,y_1,\dots,y_{k-1}) &= \int_{\mathbb{R}^n} \frac{\partial^k \Phi}{\partial \theta^k}(z, \rho * m(z)) \rho(y - z)\rho(y_1 - z) \cdots \rho(y_{k-1} - z)\,  dz.
    \end{align}
    Consequently, $F(m)$ is analytic provided the partial derivatives $\frac{\partial^k \Phi}{\partial \theta^k}(z,\theta)$ are uniformly bounded in $\mathbb{R}^n$ for $k = 1, 2, \dots, N$. This construction demonstrates how analytic functions on $\mathcal{P}(\mathbb{T}^n)$ can be built from analytic functions on $\mathbb{R}^n$, and it is straightforward to construct such functions that also satisfy  all required regularity conditions introduced on $F$ and $G$ at the beginning of this section.  
\end{remark}

We proceed to establish the injectivity of $\mathcal{M}$ in \eqref{eq:dpem1}. First, we show that it is necessary to require that $F$ is time-independent. In fact, we have the following lemma.  

\begin{lemma}[Non-uniqueness for time-dependent running costs]\label{nonunique-time}
    There exist time-dependent running costs $F_1(t,m) \neq F_2(t,m)$ and a terminal cost $G \in \mathcal{U}$ such that $\mathcal{M}(F_1) = \mathcal{M}(F_2)$.
\end{lemma}

\begin{proof}
    Consider the case $G(m) = 0$ and let $F_i(t,m) = g_i(t)$ for some real-valued functions $g_i$. The DPE becomes 
    \begin{equation}\label{simple case}
        \begin{cases}
            -\partial_t U_i(t,m) + \mathcal{L}\left(\frac{\delta U_i}{\delta m}, m\right) = g_i(t), \\
            U_i(T,m) = 0,
        \end{cases}  
    \end{equation}
    where
    \begin{equation*}
        \mathcal{L}\left(\frac{\delta U}{\delta m}, m\right) = \int_{\mathbb{T}^n} \operatorname{div}(D_m U)(t,y)m(y)\,  dy + \int_{\mathbb{T}^n} \nabla f \cdot D_m U(t,y) m(y)\,  dy.
    \end{equation*}
    
    A direct verification shows that the solutions are given by 
    \begin{equation*}
        U_i(t,m) = \int_0^{T-t} g_i(s)  ds.
    \end{equation*}
    Therefore, if we choose $g_1$ and $g_2$ such that $\int_0^T g_1(s)  ds = \int_0^T g_2(s)  ds$ but $g_1 \neq g_2$, then $\mathcal{M}(F_1) = \mathcal{M}(F_2)$ while $F_1(t,m) \neq F_2(t,m)$.
\end{proof}


Next, we establish the global unique identifiability for the inverse problem \eqref{eq:dpem1}, namely the injectivity of $\mathcal{M}$, with a time-independent running cost $F$, employing the progressive variational scheme.

\begin{theorem}[Uniqueness for time-independent running costs]\label{unique-time-indep}
    Suppose $F(m)$ is time-independent and analytic with respect to the measure variable $m$. Then the measurement map $\mathcal{M}$ is injective.
\end{theorem}

\begin{proof}
 Consider the following two systems ($i=1,2$), 
 \begin{equation}\label{general case 2}
\begin{cases}
    -\partial_t U_i(t,m)+\mathcal{L}(\frac{\delta U_i}{\delta m},m)=F_i(m),\\
    U_i(T,m)=G(m),
\end{cases}  
\end{equation}
where
\begin{equation*}
     \mathcal{L}\left(\frac{\delta U}{\delta m},m\right)=\int_{\mathbb{T}^n} \text{div}(D_m U)(t,y)m(y)\,dy+\int_{\mathbb{T}^n} \nabla f\cdot D_m U(t,y)m(y)\,dy.
\end{equation*}
\textbf{First-order terms:}
Suppose $U_1(0,m_0) = U_2(0,m_0)$ for all $m_0 \in \mathcal{P}(\mathbb{T}^n)$, with $F_i$ independent of $t$. Fix $m_0 \in \mathcal{P}(\mathbb{T}^n) \cap C^\alpha(\mathbb{T}^n)$ (viewing $m_0$ as its density). Since $U_i$ ($i=1,2$) are $C^1$ in the measure variable, we have 
    \begin{equation}
        \frac{\delta U_1}{\delta m} (0,m_0,y) = \frac{\delta U_2}{\delta m}(0,m_0,y).
    \end{equation}

    Consider the corresponding auxiliary systems for $t = 0$, 
    \begin{equation}\label{corr.sys.2}
        \begin{cases}
           \displaystyle{ -\frac{d}{dt} u_i(t) = F_i(m_i), }& \text{in } (0,T),\smallskip \\
            \partial_t m_i - \Delta m_i - \nabla \cdot (m_i \nabla f) = 0, & \text{in } (0,T) \times \mathbb{T}^n, \\
            m_i(0,\cdot) = m_0(\cdot), & \text{in } \mathbb{T}^n, \\
            u_i(T) = G(m_i(T,\cdot)), & \text{in } \mathbb{T}^n, 
        \end{cases}  
    \end{equation}
    and their linearizations around $(u_i, m_i)$, 
    \begin{equation}\label{corr.sys.linear.2}
        \begin{cases}
           \displaystyle{ -\frac{d}{dt}v_i(t) = \frac{\delta F_i}{\delta m}(m_i(t))(\mu_i),} & \text{in } (0,T), \smallskip\\
            \partial_t \mu_i - \Delta \mu_i - \nabla \cdot (\mu_i \nabla f) = 0, & \text{in } (0,T) \times \mathbb{T}^n, \\
            \mu_i(0) = \mu_0, & \text{in } \mathbb{T}^n,\smallskip \\
            \displaystyle{v_i(T) = \frac{\delta G}{\delta m}(m_i(T))(\mu_i(T)),} & \text{in } \mathbb{T}^n.
        \end{cases}  
    \end{equation}

    By equation \eqref{key} and the results of Subsection \ref{differentiability of U}, we obtain 
    \begin{equation}
        v_1(0) = \int_{\mathbb{T}^n} \frac{\delta U_1}{\delta m}(0,m_0,y)\mu_0(y)  dy = \int_{\mathbb{T}^n} \frac{\delta U_2}{\delta m}(0,m_0,y)\mu_0(y)  dy = v_2(0).
    \end{equation}

   Since $m_i$ and $\mu_i$ share the same initial condition 
   (equation ~\eqref{corr.sys.2} and ~\eqref{corr.sys.linear.2}), we define $m_1=m_2\equiv\hat{m}$ and $\mu_1=\mu_2\equiv\mu$. Let $\overline{v}(t) \equiv v_1(t) - v_2(t)$. Moreover, we have 
    \begin{equation}
        v_i(t) = \int_0^{T-t} \frac{\delta F_i}{\delta m}(\hat{m})(\mu)  ds + \frac{\delta G}{\delta m}(\hat{m}(T))(\mu(T)).
    \end{equation}
    Therefore,
    \begin{equation}\label{integral-eq}
        \int_0^T \left[ \frac{\delta F_1}{\delta m}(\hat{m})(\mu) - \frac{\delta F_2}{\delta m}(\hat{m})(\mu) \right] ds = 0.
    \end{equation}

    Let $\lambda$ be an eigenvalue of the operator $-\Delta + \nabla f \cdot \nabla + \Delta f$ with $E_\lambda(x)$ the corresponding eigenfunction. Choosing $\mu_0(x) = E_\lambda(x)$, the solution to the linearized equation is $\mu(t,x) = e^{-\lambda t} E_\lambda(x)$. Further selecting $m_0(x) = \frac{e^{-f(x)}}{\int_{\mathbb{T}^n} e^{-f(y)}  dy}$, we can verify that $\hat{m}(t,x) = m_0(x)$.

    In this case, equation \eqref{integral-eq} implies that  
    \begin{equation}\label{eq:DPE-result}
        \int_{\mathbb{T}^n} \left[ \frac{\delta F_1}{\delta m}(\hat{m}, y) - \frac{\delta F_2}{\delta m}(\hat{m}, y) \right] E_\lambda(y)\,  dy = 0.
    \end{equation}

    Since $-\Delta + \nabla f \cdot \nabla + \Delta f$ is self-adjoint in $L^2(\mathbb{T}^n, e^{-f}dx)$, the eigenfunctions $\{E_\lambda\}$ form a complete basis. Since, by Theorem \ref{diff_wrt_m}, $\frac{\delta F_1}{\delta m}(\hat{m}, y) - \frac{\delta F_2}{\delta m}(\hat{m}, y)$ is continuous in $y$, ~\eqref{eq:DPE-result} implies that 
    \begin{equation*}
        \frac{\delta F_1}{\delta m}(\hat{m}, y) = \frac{\delta F_2}{\delta m}(\hat{m}, y), \quad \forall y \in \mathbb{T}^n.
    \end{equation*}

    \textbf{Higher-order terms:} Applying the second-order mixed directional derivative $\partial^2_{s_1 s_2}$ (with direction $(\mu_0, \mu_0)$) to system \eqref{eq:mfc} and defining $U_i^{(2)}(t,m) \equiv \partial^2_{s_1 s_2} U_i(t,m)$ for $i=1,2$, we find that $U_i^{(2)}$ satisfies 
    \begin{equation*}
        \begin{cases}
          \displaystyle{  -\partial_t U_i^{(2)}(t,m) + \mathcal{L}\left(\frac{\delta U_i^{(2)}}{\delta m}, m\right) = \int_{\mathbb{T}^n} \int_{\mathbb{T}^n} \frac{\delta^2 F_i}{\delta m^2}(m,y,y') \mu_0(y)\mu_0(y')\,  dy dy',}\smallskip \\
          \displaystyle{  U_i^{(2)}(T,m) = \int_{\mathbb{T}^n} \int_{\mathbb{T}^n} \frac{\delta^2 G}{\delta m^2}(m,y,y') \mu_0(y)\mu_0(y')\,  dy dy'.}
        \end{cases}
    \end{equation*}
    Since $U_1^{(2)}(0,m) = U_2^{(2)}(0,m)$ by definition, repeating the previous argument yields 
    \begin{equation}
        \frac{\delta^2 F_1}{\delta m^2}(\hat{m}, y_1, y_2) = \frac{\delta^2 F_2}{\delta m^2}(\hat{m}, y_1, y_2), \quad \forall y_1, y_2 \in \mathbb{T}^n.
    \end{equation}

    Proceeding by induction, for any $k \geq 1$, applying the $k$-th order mixed directional derivative $\partial^k_{s_1 \cdots s_k}$ (with direction $(\mu_0, \dots, \mu_0)$) gives 
    \begin{equation}
        \frac{\delta^k F_1}{\delta m^k}(\hat{m}, y, y_1, \dots, y_{k-1}) = \frac{\delta^k F_2}{\delta m^k}(\hat{m}, y, y_1, \dots, y_{k-1}).
    \end{equation}

    \textbf{Zeroth-order term:} For $k = 0$, the condition $U_1(0, \hat{m}) = U_2(0, \hat{m})$ implies that 
    \begin{equation}
        G(\hat{m})+\int_0^T F_1(\hat{m})\,  ds = u_1(0) = u_2(0) = G(\hat{m})+\int_0^T F_2(\hat{m})\,  ds.
    \end{equation}
    Since $\hat{m}$ is time-independent, we conclude that $F_1(\hat{m}) = F_2(\hat{m})$.

    \textbf{Conclusion:} Given that $F_1$ and $F_2$ are analytic with respect to $m$ and agree on all derivatives at $\hat{m}$, we conclude that 
    \begin{equation*}
        F_1(m) = F_2(m), \quad \forall m \in \mathcal{P}(\mathbb{T}^n).
    \end{equation*}
\end{proof}


\section{Inverse Problems for the MFG Master Equation in $\mathbb{T}^{n}$}
\label{IPMFGMaster}
\subsection{Introduction to the Master Equation and static solutions}

The theory of Mean Field Games (MFGs) has emerged as a powerful framework for modeling strategic interactions among large populations of rational agents. In the MFG paradigm, each agent solves an individual optimal control problem that is coupled through an aggregate population distribution, leading to a system of coupled partial differential equations: a Hamilton-Jacobi-Bellman (HJB) equation describing optimal individual behavior, and a Fokker-Planck (FP) equation governing the evolution of the population distribution.

To provide a more comprehensive and unified perspective on MFG equilibria, Lasry and Lions introduced the \emph{Master Equation}. This infinite-dimensional partial differential equation is formulated on the space $[0,T] \times \mathbb{R}^n \times \mathcal{P}(\mathbb{R}^n)$ and encodes the value function of a generic agent while fully characterizing the game's equilibrium. The Master Equation serves as a fundamental object that encapsulates the entire dynamics of the MFG system, allowing for a more direct analysis of convergence, stability, and long-term behavior.

The Master Equation takes the following general form, 
\begin{equation}\label{master_equation}
    \begin{cases}
        -\partial_t U(t,x,m) - \Delta_x U(t,x,m) + H\big(x, D_x U(t,x,m)\big) \\
        \quad - \displaystyle\int_{\mathbb{T}^n} \nabla_y \cdot \big[D_m U\big](t,x,m,y)\, dm(y) \\
        \quad + \displaystyle\int_{\mathbb{T}^n} D_m U(t,x,m,y) \cdot H\big(y, D_x U(t,y,m)\big)\, dm(y) 
        =  \mathscr{F}(x,m), \\
        \qquad \text{in } (0,T) \times \mathbb{T}^n \times \mathcal{P}(\mathbb{T}^n), \\[2mm]
        U(T,x,m) = \mathscr{G}(x,m), \\
        \qquad \text{in } \mathbb{T}^n \times \mathcal{P}(\mathbb{T}^n).
    \end{cases}
\end{equation}
where the value function \( U(t, x, m) \) represents the optimal value for a typical agent at time \( t \), with state \( x \), and under the population distribution \( m \). The function \( \mathscr{F} \) models the running cost, capturing the agent's interaction with the population, while \( \mathscr{G} \) denotes the terminal cost and $H$ denotes the Hamiltonian.

A crucial concept in the analysis of MFGs is that of \emph{static solutions} (see Definition~\ref{static-solution}). These represent long-term equilibrium states where the population distribution and value function become time-independent. Static solutions provide fundamental insights into the long-term behavior of MFG systems and serve as natural reference states for studying perturbations and inverse problems.

In this section, we investigate inverse problems for the Master Equation defined on the torus $\mathbb{T}^n$. We develop a systematic approach for recovering unknown structures—such as the running cost $\mathscr{F}(x,m)$ and Hamiltonian $H(x,p)$—from observational data of the system's solutions. By combining progressive variational   methods with static solution analysis, we establish identifiability results that reveal how the intrinsic governing laws of MFG systems can be reconstructed from their emergent collective behavior.

Cardaliaguet, Delarue, Lasry and Lions \cite{Linear_in_Tn} have developed a comprehensive theory concerning the well-posedness of the Master Equation. Below, we recall some essential definitions and results that will be used.

\begin{definition}
    A map $U: \mathbb{R} \times \mathbb{T}^n \times \mathcal{P}(\mathbb{T}^n) \to \mathbb{R}$ is called a \emph{classical solution} of the first-order Master Equation if the following conditions are satisfied: 

    \begin{enumerate}
        \item $U$ is continuous in all variables (with respect to the $\mathfrak{D}$ metric), where for each fixed $m \in \mathcal{P}(\mathbb{T}^n)$, the function $U(\cdot, \cdot, m)$ belongs to $C^{1,2}(\mathbb{R} \times \mathbb{T}^n)$; 
        
        \item $U$ is continuously differentiable with respect to the measure variable $m$, where the first-order derivative $\frac{\delta U}{\delta m}(t,x,m,y)$ is continuous in all arguments and twice differentiable in $y$, with all derivatives continuous in all variables; 
        
        \item $U$ satisfies the Master Equation \eqref{master_equation}.
    \end{enumerate}
\end{definition}

To establish the well-posedness of the Master Equation, we require certain regularity conditions on the functions $\mathscr{F}$ and $\mathscr{G}$. 
    \begin{enumerate}
        \item \textbf{Regularity conditions:}
            \begin{align}
                \sup_{m \in \mathcal{P}(\Omega)} \left( \|\mathscr{F}(\cdot, m)\|_\alpha + \left\|\frac{\delta \mathscr{F}}{\delta m}(\cdot, m, \cdot)\right\|_{\alpha, 2+\alpha} \right) + \text{Lip}\left(\frac{\delta \mathscr{F}}{\delta m}\right) &< \infty, \\
                \sup_{m \in \mathcal{P}(\Omega)} \left( \|\mathscr{G}(\cdot, m)\|_\alpha + \left\|\frac{\delta \mathscr{G}}{\delta m}(\cdot, m, \cdot)\right\|_{\alpha, 2+\alpha} \right) + \text{Lip}\left(\frac{\delta \mathscr{G}}{\delta m}\right) &< \infty,
            \end{align}
            where the Lipschitz constants are defined by 
            \begin{align}
                \text{Lip}\left(\frac{\delta \mathscr{F}}{\delta m}\right) &\equiv \sup_{m_1 \neq m_2} \left[ \mathfrak{D}(m_1, m_2)^{-1} \left\|\frac{\delta \mathscr{F}}{\delta m}(\cdot, m_1, \cdot) - \frac{\delta \mathscr{F}}{\delta m}(\cdot, m_2, \cdot)\right\|_{\alpha, 1+\alpha} \right], \\
                \text{Lip}\left(\frac{\delta \mathscr{G}}{\delta m}\right) &\equiv \sup_{m_1 \neq m_2} \left[ \mathfrak{D}(m_1, m_2)^{-1} \left\|\frac{\delta \mathscr{G}}{\delta m}(\cdot, m_1, \cdot) - \frac{\delta \mathscr{G}}{\delta m}(\cdot, m_2, \cdot)\right\|_{\alpha, 1+\alpha} \right].
            \end{align}
        
        \item \textbf{Higher-order regularity:} For all $k \in \mathbb{N}$,
            \begin{align}
                \sup_{m \in \mathcal{P}(\Omega)} \left\| \frac{\delta^k \mathscr{F}}{\delta m^k}(\cdot, m, y, y_1, \dots, y_{k-1}) \right\|_{\alpha, (2+\alpha)^k}+\mathrm{Lip}\left(\frac{\delta^k \mathscr{F}}{\delta m^k}\right) &< \infty, \\
                \sup_{m \in \mathcal{P}(\Omega)} \left\| \frac{\delta^k \mathscr{G}}{\delta m^k}(m, y, y_1, \dots, y_{k-1}) \right\|_{(2+\alpha)^k}+\mathrm{Lip}\left(\frac{\delta^k \mathscr{G}}{\delta m^k}\right) &< \infty.
            \end{align}
    \end{enumerate}
Here, $\Omega$ may denote either $\mathbb{T}^n$ or a general bounded domain in $\mathbb{R}^n$. We employ this general notation as these assumptions will be applied to problems in various domains throughout the paper.
We also require coercivity conditions on the Hamiltonian $H$ and monotonicity conditions on the functions $\mathscr{F}$ and $\mathscr{G}$.
    \begin{enumerate}
        \item The Hamiltonian $H: \mathbb{T}^n \times \mathbb{R}^n \to \mathbb{R}$ is smooth and globally Lipschitz continuous, and satisfies the uniform ellipticity condition, 
            \begin{equation}
                0 < D^2_{pp} H(x,p) \leq C I_n, \quad \text{for all } (x,p) \in \mathbb{T}^n \times \mathbb{R}^n,
            \end{equation}
            where $I_n$ denotes the $n \times n$ identity matrix.
            
        \item The functions $\mathscr{F}, \mathscr{G}: \mathbb{T}^n \times \mathcal{P}(\mathbb{T}^n) \to \mathbb{R}$ are globally Lipschitz continuous and satisfy the following monotonicity conditions: for any $m, m' \in \mathcal{P}(\mathbb{T}^n)$,
            \begin{equation}
                \begin{aligned}
                    &\int_{\mathbb{T}^n} \left(\mathscr{F}(x, m) - \mathscr{F}(x, m')\right) d(m - m')(x) \geq 0, \\
                    &\int_{\mathbb{T}^n} \left(\mathscr{G}(x, m) - \mathscr{G}(x, m')\right) d(m - m')(x) \geq 0.
                \end{aligned}
            \end{equation}
    \end{enumerate}
We maintain these conditions throughout the section, and we further present some key results from \cite{Linear_in_Tn}, reformulated for our application.

\begin{theorem}[Theorem 2.8 in \cite{Linear_in_Tn}]
   The first-order Master Equation \eqref{master_equation} admits a unique classical solution $U$. Moreover, the following hold: 
    \begin{enumerate}
        \item $U$ is continuously differentiable in all variables;
        \item $\frac{\delta U}{\delta m}$ is continuous in all arguments;
        \item $U(t,\cdot,m)$ and $\frac{\delta U}{\delta m}(t,\cdot,m,\cdot)$ are uniformly bounded in $C^{2+\alpha}$ and $C^{2+\alpha,1+\alpha}$, respectively, independent of $t$ and $m$.
    \end{enumerate}
\end{theorem}

We now introduce the corresponding mean field game (MFG) system. 

\begin{theorem}[Section 2.4 in \cite{Linear_in_Tn}]
    Let $(u,m)$ be the solution of the following MFG system, 
    \begin{equation}\label{mfg}
        \begin{cases}
            -\partial_t u(t,x) - \Delta u(t,x) + H(x,\nabla u) = \mathscr{F}(x,m(t)), & \text{in } (t_0,T)\times\mathbb{T}^n , \\
            \partial_t m(t,x) - \Delta m(t,x) - \mathrm{div}\left(m(t,x) H_p(x,\nabla u)\right) = 0, & \text{in } (t_0,T)\times\mathbb{T}^n, \\
            u(T,x) = \mathscr{G}(x,m(T)), & \text{in } \mathbb{T}^n, \\
            m(t_0,x) = m_0(x), & \text{in } \mathbb{T}^n.
        \end{cases}  		
    \end{equation}
    Define the value function $V(t_0,x,m_0) \equiv u(t_0,x; m_0)$. Then $V$ coincides with the solution $U$ of the Master Equation \eqref{master_equation}.
\end{theorem}

Fix $t_0 \in [0,T]$ and $m_0 \in \mathcal{P}(\mathbb{T}^n)$, and let $(u,m)$ be the solution to the MFG system with the initial condition $m(t_0,\cdot) = m_0$. Consider the associated linearized system,
\begin{equation}\label{linearized system}
    \begin{cases}
        -\partial_t v - \Delta v + H_p(x,\nabla u) \cdot \nabla v = \dfrac{\delta \mathscr{F}}{\delta m} (x,m(t))(\rho(t)), & \text{in } (t_0,T) \times \mathbb{T}^n, \\
        \partial_t \rho - \Delta \rho - \nabla \cdot \left(m D^2_{pp}H(x,\nabla u) \nabla v\right) - \nabla \cdot \left(\rho H_p(x,\nabla u)\right) = 0, & \text{in } (t_0,T) \times \mathbb{T}^n, \\
        v(T,x) = \dfrac{\delta \mathscr{G}}{\delta m}(x,m(T))(\rho(T)), & \text{in } \mathbb{T}^n, \\
        \rho(t_0,x) = \mu_0(x), & \text{in } \mathbb{T}^n.
    \end{cases}
\end{equation}

\begin{theorem}[Section 3.3 in \cite{Linear_in_Tn}]
    Let $m_0$ be a smooth density bounded below by a positive constant, and let $\mu_0$ be a smooth function on $\mathbb{T}^n$. Then the linearized system \eqref{linearized system} admits a unique solution $(v, \rho)$.
\end{theorem}

The connection between the derivative $\frac{\delta U}{\delta m}$ and the linearized system is established by the following result \cite{Linear_in_Tn}. 

\begin{theorem}[Section 3.4 in \cite{Linear_in_Tn}]\label{Theorem-key2}
    Let $U(t,x,m)$ be the solution of the Master Equation \eqref{master_equation}, and let $v(t_0,x)$ be the solution of the linearized system \eqref{linearized system} with the initial condition $\rho(t_0,\cdot) = \mu_0$. We have 
    \begin{equation}\label{key2}
        v(t_0,x) = \int_{\mathbb{T}^n} \dfrac{\delta U}{\delta m} (t_0,x,m_0,y) \mu_0(y)  dy.
    \end{equation}
\end{theorem}
Finally, we define the static solution of the MFG system. 

\begin{definition}\label{static-solution}
    A pair $(u^{(0)}, m^{(0)})$ is called a \emph{static solution} of the MFG system \eqref{mfg} if it satisfies
    \begin{equation*}
        \begin{cases}
            \gamma - \Delta u^{(0)}(x) + H(x, \nabla u^{(0)}) = \mathscr{F}(x, m^{(0)}(x)), & \text{in } \mathbb{T}^n,\smallskip \\
            -\Delta m^{(0)}(x) - \operatorname{div}\left(m^{(0)}(x) H_p(x, \nabla u^{(0)})\right) = 0, & \text{in } \mathbb{T}^n,\smallskip\\
         \displaystyle{   \int_{\mathbb{T}^n}m^{(0)}(x)\,dx=1,} 
        \end{cases}  
    \end{equation*}
    for some constant $\gamma \in \mathbb{R}$.
    
    \smallskip
    Note that if $(u^{(0)}, m^{(0)})$ is a static solution of \eqref{mfg}, then the pair $(u, m) = (u^{(0)}(x) + \lambda(T-t), m^{(0)})$ is a solution of \eqref{mfg} for an appropriately chosen initial data $m_0(x)$ and a terminal cost $\mathscr{G}$. In this case, we say that $(u, m)$ is a \emph{quasi-static solution} of \eqref{mfg}.
\end{definition}
The existence of such static solutions under our regularity assumptions is established in \cite{cardaliaguet2015master}.

\subsection{Auxiliary results}

In this subsection, we establish several auxiliary results that will be essential for our study of inverse problems. Recall that we interpret $\frac{\delta \mathscr{F}}{\delta m}(x,m)$ as an operator defined by 
\begin{equation}
    \frac{\delta \mathscr{F}}{\delta m}(x,m)(\mu) \equiv \int_{\mathbb{T}^n} \frac{\delta \mathscr{F}}{\delta m}(x,m,y) \mu(y)  dy.
\end{equation}
We also define $Q\equiv(0,T)\times\mathbb{T}^n.$
\begin{lemma}\label{well_pose}
   Let $(u^{(0)},m^{(0)})$ be a static solution of $\eqref{mfg}$. Consider two solutions $(u_1, m_1)$ and $(u_2, m_2)$ of the linear system, 
    \begin{equation}\label{1st order_for_unique}
        \begin{cases}
            -\partial_t u - \Delta u + H_p(x, \nabla u^{(0)}) \cdot \nabla u = \dfrac{\delta \mathscr{F}}{\delta m} (x, m^{(0)})(m),  &\text{ in } (0,T)\times\mathbb{T}^n,\\
            \partial_t m - \Delta m - \nabla \cdot \left(m^{(0)} D^2_{pp}H(x, \nabla u^{(0)}) \nabla u\right) - \nabla \cdot \left(m H_p(x, \nabla u^{(0)})\right) = 0, &\text{ in } (0,T)\times\mathbb{T}^n,\\
            m(0, x) = f(x),   &\text{ in } \mathbb{T}^n, \\
            u(T, x) = g(x),  &\text{ in } \mathbb{T}^n.
        \end{cases}
    \end{equation}
    Then $(u_1, m_1) = (u_2, m_2)$ so that the solution of this system is unique.
\end{lemma}

\begin{proof}
Since the system is linear, it suffices to show that $(u, m) = (0, 0)$ when $(f, g) = (0, 0)$. Recall that $Q=(0,T)\times\mathbb{T}^n$. We carry out the following computation, 
    \begin{equation*}
        \begin{aligned}
          &  \int_{Q} \left(\partial_t u  m + \partial_t m  u\right) dx  dt \\
            =&  \int_{Q} \left(-\Delta u + H_p(x, \nabla u^{(0)}) \cdot \nabla u - \dfrac{\delta \mathscr{F}}{\delta m} (x, m^{(0)})(m)\right) m  dx  dt \\
            &\quad + \int_{Q} \left(\Delta m + \nabla \cdot \left[m^{(0)} D^2_{pp}H(x, \nabla u^{(0)}) \nabla u\right] + \nabla \cdot \left(m H_p(x, \nabla u^{(0)})\right)\right) u  dx  dt \\
            =& \int_{Q} \left[-\dfrac{\delta \mathscr{F}}{\delta m} (x, m^{(0)})(m) m + (-\Delta u  m + \Delta m  u)\right] dx  dt \\
            &\quad + \int_{Q} \bigg[m H_p(x, \nabla u^{(0)}) \cdot \nabla u + \nabla \cdot \left(m H_p(x, \nabla u^{(0)})\right) u\\
            & + \nabla \cdot \left[m^{(0)} D^2_{pp}H(x, \nabla u^{(0)}) \nabla u\right] u\bigg] dx  dt \\
            =&  \int_{Q} \left[-\dfrac{\delta \mathscr{F}}{\delta m} (x, m^{(0)})(m) m - m^{(0)} D^2_{pp}H(x, \nabla u^{(0)}) |\nabla u|^2\right] dx  dt.
        \end{aligned}
    \end{equation*}

    If $(f, g) = (0, 0)$, then we have 
    \begin{equation}
        \int_{Q} \left(\partial_t u  m + \partial_t m  u\right) dx  dt = \left(\int_{\mathbb{T}^n} u(t,x) m(t,x)  dx \right)\Big|_0^T = 0.
    \end{equation}
    Moreover, the second equation in system \eqref{1st order_for_unique} implies that 
    \begin{equation}
        \int_{\mathbb{T}^n} m(t,x)  dx = 0 \quad \text{for all } t \in (0,T).
    \end{equation}
    Therefore, we obtain 
    \begin{equation*}
        \begin{aligned}
            &\int_{Q} \dfrac{\delta \mathscr{F}}{\delta m} (x, m^{(0)})(m) m(t,x)  dx  dt \\
            =& \int_0^T \int_{\mathbb{T}^n} \lim_{s \to 0^+} \left[ \int_{\mathbb{T}^n} \frac{1}{s} \mathscr{F}(x, m^{(0)} + s m) m(t,y)  dy \right] m(t,x)  dx  dt \\
            =& \int_0^T \int_{\mathbb{T}^n} \lim_{s \to 0^+} \left[ \int_{\mathbb{T}^n} \frac{1}{s} \left(\mathscr{F}(x, m^{(0)} + s m) - \mathscr{F}(x, m^{(0)})\right) m(t,x)  dy \right] m(t,x)  dx  dt \\
            \geq &\, 0,
        \end{aligned}
    \end{equation*}
    where the inequality follows from the monotonicity property. Since $m^{(0)}(x) \geq 0$ and $D^2_{pp}H(x, \nabla u^{(0)})$ is positive definite, we conclude that $m = 0$ and consequently $u = 0$. This establishes the uniqueness of the solution.
\end{proof}
\begin{corollary}\label{Comparing principle}
    Under the same conditions as Lemma \ref{well_pose}, if $(u,m)$ satisfies the system, 
    \begin{equation*}\label{eq:Comparing principle}
        \begin{cases}
            -\partial_t u - \Delta u + H_p(x, \nabla u^{(0)}) \cdot \nabla u = \dfrac{\delta \mathscr{F}}{\delta m} (x, m^{(0)})(m), &\text{ in } (0,T)\times\mathbb{T}^n,\\
            \partial_t m - \Delta m - \nabla \cdot \left(m^{(0)} D^2_{pp}H(x, \nabla u^{(0)}) \nabla u\right) - \nabla \cdot \left(m H_p(x, \nabla u^{(0)})\right) = 0, &\text{ in } (0,T)\times\mathbb{T}^n,\\
            m(0,x) = 0, &\text{ in } \mathbb{T}^n,\\
            m(T,x) = 0,&\text{ in } \mathbb{T}^n, 
        \end{cases}
    \end{equation*}
    then $(u,m) = (0,0)$. This result follows by applying the same energy method used in the proof of Lemma \ref{well_pose}.
\end{corollary}

In the proof of our main theorem, we also need to analyze the quasi-static solutions of the linearized system. The following existence result is crucial. 

\begin{lemma}\label{quasi_static}
    Let $(u^{(0)},m^{(0)})$ be a static solution of $\eqref{mfg}$. Then there exist initial and terminal conditions $f$ and $g$ such that the solution $(u,m)$ of the system, 
    \begin{equation}\label{1st order_for_exist}
        \begin{cases}
            -\partial_t u - \Delta u + H_p(x, \nabla u^{(0)}) \cdot \nabla u = \dfrac{\delta \mathscr{F}}{\delta m} (x, m^{(0)})(m),&\text{ in } (0,T)\times\mathbb{T}^n, \\
            \partial_t m - \Delta m - \nabla \cdot \left(m^{(0)} D^2_{pp}H(x, \nabla u^{(0)}) \nabla u\right) - \nabla \cdot \left(m H_p(x, \nabla u^{(0)})\right) = 0, &\text{ in } (0,T)\times\mathbb{T}^n,\\
            m(0,x) = f(x), &\text{ in } \mathbb{T}^n,\\
            u(T,x) = g(x),&\text{ in } \mathbb{T}^n, 
        \end{cases}
    \end{equation}
    satisfies the stationarity conditions, 
    \begin{equation*}
            \dfrac{d}{dt} \nabla u(t,x) = 0, \quad
            \dfrac{d}{dt} m(t,x) = 0, 
    \end{equation*}
    and, moreover, $m(t,x)$ does not vanish identically on any open subset $\Omega \subset \mathbb{T}^n$.
\end{lemma}

\begin{proof}
    Consider the stationary system, 
    \begin{equation}\label{1st order-static}
        \begin{cases}
            \eta - \Delta v + H_p(x, \nabla u^{(0)}) \cdot \nabla v = \dfrac{\delta \mathscr{F}}{\delta m} (x, m^{(0)})(\rho), & \text{in } \mathbb{T}^n, \\
            -\Delta \rho - \nabla \cdot \left(m^{(0)} D^2_{pp}H(x, \nabla u^{(0)}) \nabla v\right) - \nabla \cdot \left(\rho H_p(x, \nabla u^{(0)})\right) = 0, & \text{in } \mathbb{T}^n, \\
            \displaystyle \int_{\mathbb{T}^n} \rho(x)  dx = 1.
        \end{cases}
    \end{equation}
    The existence of a solution $(\eta, v(x), \rho(x))$ to this system can be established via standard fixed-point arguments.

    Define the following time-dependent functions, 
    \begin{equation}
        (u(t,x), m(t,x)) = (v(x) + \eta(T - t), \rho(x)).
    \end{equation}
    A direct computation shows that this pair satisfies system \eqref{1st order_for_exist} with the initial condition $f(x) = \rho(x)$ and the terminal condition $g(x) = v(x)$. By Lemma \ref{well_pose}, this solution is unique.

    Now suppose, by contradiction, that there exists an open set $\Omega \subset \mathbb{T}^n$ such that $m(t,x) = 0$ on $\Omega$. Since $m(t,x) = \rho(x)$ is time-independent and satisfies the second equation in \eqref{1st order-static}, the unique continuation principle for elliptic equations implies that $\rho(x) \equiv 0$ on $\mathbb{T}^n$. This contradicts the normalization condition $\int_{\mathbb{T}^n} \rho(x)  dx = 1$, completing the proof.
\end{proof}
\subsection{Inverse problems for the Master Equation using time boundary data}

In this subsection, we formulate and study inverse problems for the Master Equation $\eqref{master_equation}$ using time boundary data. We define the measurement map 
\begin{equation}
    \mathcal{N}_{\mathscr{F},H}(\mathscr{G}) = (u(0, \cdot), m(T, \cdot)),
\end{equation}
 where $(u, m)$ is the solution of the corresponding MFG system \eqref{mfg} with the running cost $\mathscr{F}(x, m)$, the Hamiltonian $H$, and an arbitrary initial distribution. Since $u(0, x) = U(0, x, m)$, the map $\mathcal{N}_{\mathscr{F},H}(\mathscr{G})$ can be understood as the time boundary data of $U(t, x, m)$ in the Wasserstein space. This formulation is analogous to that of Problem~\eqref{eq:ip1}.

In Section~\ref{ipmfcdpe}, we established that if two distinct dynamic programming equations yield identical measurement data, their corresponding systems must share the same static solution. For the Master Equation, however, deriving an analogous result requires stronger assumptions.

First, Theorem~\ref{Main2} below explicitly incorporates this requirement as an assumption. This can be interpreted as follows: given an MFG system that has stabilized to an unknown stationary state $(u^{(0)}, m^{(0)})$, we introduce perturbations and measure the resulting response data. We refer to \cite{Cauchydata-zs} for an extended discussion and additional examples within this framework.

Second, Theorem~\ref{Main3} in the following demonstrates that this requirement can be relaxed under additional conditions on the unknowns. Specifically, we show that under these conditions, systems generating identical measurement data must possess the same static solution. Consequently, we establish that these conditions enable simultaneous recovery of the Hamiltonian $H$.

\begin{theorem}\label{Main2}
    Consider two Master Equations with analytic running costs $\mathscr{F}_j(x,m)$ for $j=1,2$. Suppose that 
    \begin{enumerate}
        \item both systems share the same static solution $(m^{(0)}(x), u^{(0)}(x))$;
        \item $D^2_{pp}H(x,p)$ is independent of $p$;
        \item $\mathscr{F}_i(x,m)$ for $i=1,2$ depend on $m$ locally and admit the power series expansion
        \begin{equation}
            \mathscr{F}_i(x,m) = \sum_{k=0}^{\infty} \frac{1}{k!} \mathscr{F}_i^{(k)}(x) (m - m^{(0)})^k.
        \end{equation}
     \end{enumerate}
    Then 
    \begin{equation}
        \mathcal{N}_{\mathscr{F}_1,H}(\mathscr{G}) = \mathcal{N}_{\mathscr{F}_2,H}(\mathscr{G}) \quad \text{implies that} \quad \mathscr{F}_1(x,m) = \mathscr{F}_2(x,m).
    \end{equation}
\end{theorem}

\begin{proof}
    Consider the following systems for $j = 1, 2$, 
    \begin{equation}\label{j=1,2for Fg}
        \begin{cases}
            -\partial_t u_j(t,x) - \Delta u_j(t,x) + H(x, \nabla u_j) = \mathscr{F}_j(x, m_j(t,x)), & \text{in } \mathbb{T}^n \times (0,T), \\
            \partial_t m_j(t,x) - \Delta m_j(t,x) - \operatorname{div}\left(m_j(t,x) H_p(x, \nabla u_j)\right) = 0, & \text{in } \mathbb{T}^n \times (0,T), \\
            u_j(T,x) = \mathscr{G}(x), & \text{in } \mathbb{T}^n, \\
            m_j(0,x) = m_0(x), & \text{in } \mathbb{T}^n.
        \end{cases}  		
    \end{equation}
    
    By assumption, the above two systems share the same static solution so that there exists a static solution $(m^{(0)}(x), u^{(0)}(x))$ satisfying the following two systems simultaneously, 
    \begin{equation}\label{static}
        \begin{cases}
            \gamma_j - \Delta u_j(x) + H(x, \nabla u_j) = \mathscr{F}_j(x, m_j(x)), & \text{in } \mathbb{T}^n, \\
            -\Delta m_j(x) - \operatorname{div}\left(m_j(x) H_p(x, \nabla u_j)\right) = 0, & \text{in } \mathbb{T}^n,\\
            \int_{\mathbb{T}^n} m_j(x)\,dx=1,
        \end{cases}  
    \end{equation}
    for some $\gamma_j \in \mathbb{R}$ with $j=1,2$.
    
    Choosing $m_j(0,x) = m^{(0)}(x)$ and $u_j(T,x) = u^{(0)}(x)$, the solutions of systems~\eqref{j=1,2for Fg} are given by 
    \begin{equation}
        (u_j, m_j) = (u^{(0)}(x) + \gamma_j(T - t), m^{(0)}(x)).
    \end{equation}

    Now consider the linearized system around the solution $(u_j, m_j) = (u^{(0)}(x) + \gamma_j(T - t), m^{(0)}(x))$. The first-order linearization system is 
    \begin{equation}\label{first order j=1,2}
        \begin{cases}
            -\partial_t u_j^{(1)} - \Delta u_j^{(1)} + H_p(x, \nabla u^{(0)}) \cdot \nabla u_j^{(1)} = \dfrac{\delta \mathscr{F}_j}{\delta m}(x, m^{(0)})(m_j^{(1)}), \\
            \partial_t m_j^{(1)} - \Delta m_j^{(1)} - \operatorname{div} \left(m_j^{(1)} H_p(x, \nabla u^{(0)})\right) - \operatorname{div} \left(m^{(0)} D^2_{pp} H(x, \nabla u_j^{(1)}) \nabla u_j^{(1)}\right) = 0, \\
            u^{(1)}(T,x) = 0, \\
            m^{(1)}(0,x) = \mu_0(x).
        \end{cases}    
    \end{equation}
    
    Define $\overline{u}(t,x) = u_1^{(1)}(t,x) - u_2^{(1)}(t,x)$ and $\overline{m}(t,x) = m_1^{(1)}(t,x) - m_2^{(1)}(t,x)$. Then we have 
    \begin{equation}\label{eq 1-eq 2}
        \begin{cases}
            -\partial_t \overline{u} - \Delta \overline{u} + H_p(x, \nabla u^{(0)}) \cdot \nabla \overline{u} - \dfrac{\delta \mathscr{F}_1}{\delta m}(x, m^{(0)})(\overline{m}) = \\
          \qquad  -\left(\dfrac{\delta \mathscr{F}_1}{\delta m}(x, m^{(0)}) - \dfrac{\delta \mathscr{F}_2}{\delta m}(x, m^{(0)})\right)(m_2^{(1)}), \\
            \partial_t \overline{m} - \Delta \overline{m} - \operatorname{div} \left(\overline{m} H_p(x, \nabla u^{(0)})\right) - \operatorname{div} \left(m^{(0)} D^2_{pp} H(x, \nabla u_1^{(1)}) \nabla \overline{u}\right) = R(t,x), \\
            \overline{u}(T,x) = 0, \\
            \overline{m}(0,x) = 0,
        \end{cases}
    \end{equation}
    where we have used our assumption that the Hessian $D^2_{pp} H(x,p)$ is independent of $p$ to get 
\[R(t,x) = \operatorname{div} \left(m^{(0)} \left(D^2_{pp} H(x, \nabla u_2^{(1)}) - D^2_{pp} H(x, \nabla u_1^{(1)})\right) \nabla u_2\right) = 0.\]
    
    Since $\mathcal{N}_{\mathscr{F}_1,H} = \mathcal{N}_{\mathscr{F}_2,H}$ by our assumption on the measurement map, we have 
    \begin{equation}
        u_1(0,x) = u_2(0,x) \quad \text{and} \quad m_1(T,x) = m_2(T,x),
    \end{equation}
    and consequently $U_1(0,x,m) = U_2(0,x,m)$.
    
    By Theorem~\ref{Theorem-key2},
    \begin{equation}
        u^{(1)}(0,x) = \int_{\mathbb{T}^n} \frac{\delta U_1}{\delta m}(0, m, y) \mu_0(y) \, dy = \int_{\mathbb{T}^n} \frac{\delta U_2}{\delta m}(0, m, y) \mu_0(y) \, dy = u^{(2)}(0,x),
    \end{equation}
    so that $\overline{u}(0,x) = 0.$ Using standard linearization methods (see \cite{liu2023inverse}), we also obtain 
        $\overline{m}(T,x) = 0.$
    
    According to Lemma~\ref{quasi_static}, there exist $m_2^{(1)}$ and $\nabla u_2^{(1)}$ that are independent of $t$, and there is no open set $\Omega \subset \mathbb{T}^n$ such that $m_2^{(1)} = 0$ in $\Omega$. Define $v(t,x) = \partial_t \overline{u}(t,x)$ and $\rho(t,x) = \partial_t \overline{m}(t,x)$. Then $(v, \rho)$ satisfies 
    \begin{equation}\label{u_t,v_t}
        \begin{cases}
            -\partial_t v - \Delta v + H_p(x, \nabla u^{(0)}) \cdot \nabla v - \dfrac{\delta \mathscr{F}_1}{\delta m}(x, m^{(0)})(\rho) = 0, \\
            \partial_t \rho - \Delta \rho - \operatorname{div} \left(\rho H_p(x, \nabla u^{(0)})\right) - \operatorname{div} \left(m^{(0)} D^2_{pp} H(x, \nabla u_1^{(1)}) \nabla v\right) = 0, \\
            \rho(0,x) = \rho(T,x) = 0.
        \end{cases}
    \end{equation}
    The condition $\rho(0,x) = \rho(T,x)$ is obtained by evaluating the second equation of~\eqref{first order j=1,2} at $t = 0$ and $t = T$.
    
    Lemma~\ref{well_pose} and Corollary~\ref{Comparing principle} imply that $\rho(t,x) = 0$, and therefore $\overline{m}(t,x)$ is independent of $t$. Since $\overline{m}(0,x) = 0$, we conclude that $\overline{m}(t,x) = 0$ for all $t$. Denote $m^{(1)}(x) \equiv m_1^{(1)} = m_2^{(1)}$.
    
    The second equation in~\eqref{eq 1-eq 2} then implies that 
    \begin{equation}\label{div=0}
        \operatorname{div} \left(m^{(0)}(x) D^2_{pp} H(x, \nabla u_1^{(1)}) \nabla_x \overline{u}(t,x)\right) = 0,
    \end{equation}
    for all $t \in (0,T)$.
  Since $D_{pp}^2H(x,p)>0$ for all $(x,p)\in\mathbb{T}^n\times\mathbb{R}^n$ and $D_{pp}^2H(x,p)$ is independent of $p$, there exists $\lambda(x)> 0$ such that
  \[
    v^TD_{pp}^2H(x,p)v\geq\lambda(x)|v|^2,
  \]
  for all $v\in\mathbb{R}^n$.
  Hence, ~\eqref{div=0} implies that
  \begin{equation*}
  \begin{aligned}
      0=&\int_{\mathbb{T}^n}\left(m^{(0)}(x) D^2_{pp} H(x, \nabla u_1^{(1)})\nabla_x \overline{u}(t,x)\right)\cdot\nabla_x\overline{u}(t,x) \,dx\\
        \geq& \int_{\mathbb{T}^n}m^{(0)}(x)\lambda(x)|\nabla_x\overline{u}(t,x)|^2\,dx
  \end{aligned}
  \end{equation*}
    Since $m^{(0)}$ is non-negative on $\mathbb{T}^n$, we have $m^{(0)}(x)|\nabla_x\overline{u}(t,x)|^2 = 0$ in $(0,T)\times\mathbb{T}^n$. 
    By the unique continuation principle, there is no open subset $\mathcal{O}\in\mathbb{T}^n$ such that $m^{(0)}(x)=0$ in $\mathcal{O}$. Therefore $\nabla_x \overline{u}(t,x)=0$ for all $(t,x)\in(0,T)\times\mathbb{T}^n$, then ~\eqref{eq 1-eq 2} implies that
    \begin{equation}\label{final-id}
        \frac{\delta \mathscr{F}_1}{\delta m}(x, m^{(0)})(m_2^{(1)}) = \frac{\delta \mathscr{F}_2}{\delta m}(x, m^{(0)})(m_2^{(1)}).
    \end{equation}
    
    Using the local dependence assumption, equation~\eqref{final-id} implies that 
    \begin{equation*}
        \left(\mathscr{F}_1^{(1)}(x) - \mathscr{F}_2^{(1)}(x)\right) m_2^{(1)}(x) = 0.
    \end{equation*}
    By the construction of $m_2^{(1)}(x)$, we conclude that 
    \begin{equation*}
        \mathscr{F}_1^{(1)}(x) = \mathscr{F}_2^{(1)}(x).
    \end{equation*}

    For higher-order terms ($k \geq 2$), we apply the $k$-th order mixed directional derivative operator $\partial^{k}_{s_1 s_2 \dots s_k}$ (with direction $(\mu_0^1,\mu_0^2,\cdots,\mu_0^k)$) to the Master Equation and consider the corresponding linearized system. Let $(u_j^{(i)}(t,x), m_j^{(i)}(t,x))$ for $j=1,2$ and $i=1,2,\cdots, k$ be solutions of the linearized system with initial data $m_j^{(i)}(0,x) = \mu_0^i(x)$ and running cost $\mathscr{F}_j(x,m)$. Since $\mathscr{F}_1^{(1)}(x) = \mathscr{F}_2^{(1)}(x)$, we can define   
    \begin{equation*}
    \begin{aligned}
            &m^{(i)}(t,x)\equiv m_1^{(i)}(t,x)=m_2^{(i)}(t,x), \quad i=1,2,\cdots, k,\\
            &u^{(i)}(t,x)\equiv u_1^{(i)}(t,x)=u_2^{(i)}(t,x), \quad i=1,2,\cdots, k,\\
            &\mathscr{F}^{(1)}\equiv\mathscr{F}_1^{(1)}(x) = \mathscr{F}_2^{(1)}(x).
    \end{aligned}
    \end{equation*}
    As an example, we define 
    \[
    u_j^{(1,2)}(t,x)\equiv\partial^2_{s_1s_2}U_j(t,x,m),\quad \text{ with direction } (\mu_0^1,\mu_0^2).
    \]
    It follows that $\left(u_j^{(1,2)}(t,x),m_j^{(1,2)}(t,x)\right)$ satisfies the second-order linearized system  
    \begin{equation}
        \begin{cases}
            -\partial_t u_j^{(1,2)} - \Delta u_j^{(1,2)} + H_p(x, \nabla u^{(0)}) \cdot \nabla u_j^{(1,2)} = \mathscr{F}^{(1)} m_j^{(1,2)} + \mathscr{F}_j^{(2)} m^{(1)} m^{(2)} + p_j(t,x), \\
            \partial_t m_j^{(1,2)} - \Delta m_j^{(1,2)} - \operatorname{div} \left(m_j^{(1,2)} H_p(x, \nabla u^{(0)})\right)\\
            \qquad - \operatorname{div} \left(m^{(0)} D^2_{pp} H(x, \cdot) \nabla u_j^{(1,2)}\right) = q_j(t,x), \\
            u^{(1,2)}(T,x) = 0, \\
            m^{(1,2)}(0,x) = 0,
        \end{cases}   
    \end{equation}
    where
    \begin{equation}
        \begin{aligned}
            p_j(t,x) &= H_p(x, \nabla u^{(1)}) \cdot H_p(x, \nabla u^{(2)}), \\
            q_j(t,x) &= \operatorname{div} \left(m^{(1)} D^2_{pp}H(x, \nabla u^{(2)})\cdot\nabla u^{(2)}\right) + \operatorname{div} \left(m^{(2)} D^2_{pp}H(x, \nabla u^{(1)})\cdot\nabla u^{(1)}\right).
        \end{aligned}
    \end{equation}

    Define $\tilde{u}(t,x) = u_1^{(1,2)}(t,x) - u_2^{(1,2)}(t,x)$ and $\tilde{m}(t,x) = m_1^{(1,2)}(t,x) - m_2^{(1,2)}(t,x)$. Since $p_j$ and $q_j$ are independent of $j$, we obtain 
    \begin{equation}\label{eq 1-eq 2,2}
        \begin{cases}
            -\partial_t \tilde{u} - \Delta \tilde{u} + H_p(x, \nabla u^{(0)}) \cdot \nabla \tilde{u} - \mathscr{F}^{(1)} \tilde{m} = -\left(\mathscr{F}_1^{(2)} - \mathscr{F}_2^{(2)}\right) m^{(1)} m^{(2)}, & \text{ in } (0,T)\times\mathbb{T}^n, \\
            \partial_t \tilde{m} - \Delta \tilde{m} - \operatorname{div} \left(\tilde{m} H_p(x, \nabla u^{(0)})\right) - \operatorname{div} \left(m^{(0)} D^2_{pp} H(x, \cdot) \nabla \tilde{u}\right) = 0, & \text{ in } (0,T)\times\mathbb{T}^n, \\
            \tilde{u}(T, x) = 0, & \text{ in } \mathbb{T}^n, \\
            \tilde{m}(0, x) = 0, & \text{ in } \mathbb{T}^n.
        \end{cases}
    \end{equation}
    
    By a similar argument, we conclude that $\mathscr{F}_1^{(2)}(x) = \mathscr{F}_2^{(2)}(x)$. Notice that all nonlinear terms ($p_j(x,t)$ and $q_j(x,t)$ here) in the higher-order system depend only on solutions of lower-order systems. By mathematical induction, we obtain 
    \begin{equation}
        \mathscr{F}_1^{(k)}(x) = \mathscr{F}_2^{(k)}(x), \quad \text{for all } k \in \mathbb{N}.
    \end{equation}
    
    Finally, since both systems share the same static solution, we have 
    \begin{equation}\label{eq:share the same static solution}
        \gamma_1 -\mathscr{F}_1^{(0)}(x) = \gamma_2 - \mathscr{F}_2^{(0)}(x),
    \end{equation}
where we only consider the zeroth order term of the Taylor expansion of $\mathscr{F}_j$
 for $j=1,2$. 
    
    On the other hand, choosing $m_j(0,x) = m^{(0)}(x)$ and $u_j(T,x) = u^{(0)}(x)$ for $j=1,2$, the measurement map implies that $\gamma_1 T = \gamma_2 T$. Therefore, ~\eqref{eq:share the same static solution} implies that
    \begin{equation*}
        \mathscr{F}_1^{(0)}(x) = \mathscr{F}_2^{(0)}(x).
    \end{equation*}

Since all the coefficients of the Taylor expansions of $\mathscr{F}_j$ for $j=1,2$ agree, we conclude that $\mathscr{F}_1(x,m) = \mathscr{F}_2(x,m)$.
\end{proof}
\begin{remark}
    By introducing the derivative $\frac{\delta U}{\delta m}$ of the solution for the Master Equation, we prove that the initial perturbation $\mu_0(x)$ can be chosen as an arbitrary function (Theorem~\ref{Theorem-key2}). Without this approach, the condition $\int_{\mathbb{T}^n} \mu_0(x) \, dx = 0$ would be required due to the probability measure constraint.
\end{remark}

Now consider the case where $H(x,p) = \frac{1}{2} \kappa(x) |p|^2$ for some positive function $\kappa\in C^2(\mathbb{T}^n)$. In this setting, we can also recover the Hamiltonian $H$.

\begin{theorem}\label{Main3}
    Consider two Master Equations with Hamiltonians $H_j(x,p)$ and analytic running costs $\mathscr{F}_j(x,m)$ for $j=1,2$. Suppose that 
    \begin{enumerate}
        \item $H_j(x,p) = \frac{1}{2} \kappa_j(x) |p|^2$ for some positive functions $\kappa_j\in C^2(\mathbb{T}^n)$;
        \item $\mathscr{F}_1(x, m_1^{(0)}(x)) = \mathscr{F}_2(x, m_2^{(0)}(x))$ and $\frac{\delta \mathscr{F}_j}{\delta m}(x, m_j^{(0)}(x)) = 0$ for $j=1,2$, where $m_j^{(0)}(x)$ is the static solution of the corresponding system.
    \end{enumerate}
    Then 
    \begin{equation}
        \mathcal{N}_{\mathscr{F}_1,H_1}(\mathscr{G}) = \mathcal{N}_{\mathscr{F}_2,H_2}(\mathscr{G}) \quad \text{implies that} \quad \kappa_1(x) = \kappa_2(x).
    \end{equation}
    If we further assume that $\mathscr{F}_i(x,m)$ for $i=1,2$ depend on $m$ locally and admit the expansion
    \begin{equation}
        \mathscr{F}_i(x,m) = \sum_{k=0}^{\infty} \frac{1}{k!} \mathscr{F}_i^{(k)}(x) (m - m^{(0)})^k,
    \end{equation}
    then $\mathscr{F}_1(x,m) = \mathscr{F}_2(x,m)$.
\end{theorem}

\begin{proof}
Let $(u_j^{(0)}, m_j^{(0)})$ be the static solution of the corresponding system for $j = 1, 2$. Following the same argument as in the proof of Theorem~\ref{Main2} and using the same notation, we find that $\overline{u}(t,x)$ satisfies 
\begin{equation}\label{eq 1-eq 2-H}
\begin{cases}
  -\partial_t \overline{u} - \Delta \overline{u} + \kappa_1(x) \nabla u_1^{(0)} \cdot \nabla \overline{u} 
  = -\left(\kappa_1(x) \nabla u_1^{(0)} - \kappa_2(x) \nabla u_2^{(0)}\right) \cdot \nabla u_2^{(1)}, \\
  \overline{u}(T,x) = 0.
\end{cases}
\end{equation}

Define $v(t,x) = \partial_t \overline{u}(t,x)$ and $\rho(t,x) = \partial_t \overline{m}(t,x)$. Since $\mathcal{N}_{\mathscr{F}_1,H_1}(\mathscr{G}) = \mathcal{N}_{\mathscr{F}_2,H_2}(\mathscr{G})$, an argument similar to the previous proof shows that $v$ satisfies 
\begin{equation}\label{u_t,v_t-2}
\begin{cases}
  -\partial_t v - \Delta v + \kappa_1(x) \nabla u_1^{(0)} \cdot \nabla v = 0, \\
  v(T,x) = 0.
\end{cases}
\end{equation}
It follows that $v(t,x) = 0$, and hence $\overline{u}(t,x)$ is independent of time. Since $\overline{u}(T,x) = 0$, we conclude that $\overline{u}(t,x) = 0$ for all $t \in [0,T]$.

Therefore, equation~\eqref{eq 1-eq 2-H} implies that 
\begin{equation}
  \kappa_1(x) \nabla u_1^{(0)} = \kappa_2(x) \nabla u_2^{(0)}.
\end{equation}

Now considering the static equations~\eqref{static}, we have, for $j=1,2$, 
\begin{equation}
  \lambda_j - \Delta u_j^{(0)} + \frac{1}{2} \kappa_j(x) |\nabla u_j^{(0)}|^2 = \mathscr{F}_j(x, m_j^{(0)}).
\end{equation}
Let $\frac{1}{2}\kappa_1(x) \nabla u_1^{(0)} = \frac{1}{2}\kappa_2(x) \nabla u_2^{(0)} = \vec{q}(x)$ and $\mathscr{F}_1(x, m_1^{(0)}) = \mathscr{F}_2(x, m_2^{(0)}) = p(x)$. Then $(\lambda_j, u_j^{(0)})$ for $j=1,2$ are solutions to the following equation of $(\lambda,\omega)$, 
\begin{equation}\label{eq:omega}
  \lambda - \Delta \omega(x) + \vec{q}(x) \cdot \nabla \omega(x) = p(x), \quad \text{in } \mathbb{T}^n.
\end{equation}
The solution $(\lambda, \omega)$ of equation~\eqref{eq:omega} is unique up to an additive constant for $\omega$. Therefore, we obtain 
\begin{equation*}
  \lambda_1 = \lambda_2 \quad \text{and} \quad \nabla u_1^{(0)} = \nabla u_2^{(0)}.
\end{equation*}
Hence, $\kappa_1(x) = \kappa_2(x)$.

Applying Theorem~\ref{Main2}, whose conditions are now fulfilled, we conclude that $\mathscr{F}_1(x,m) = \mathscr{F}_2(x,m)$.
\end{proof}

\section{Inverse Boundary Problems for the Master Equation}
\label{bvpmaster}

We consider the Master Equation defined on a bounded domain in $\mathbb{R}^n$. Let $\Omega \subset \mathbb{R}^n$ be a bounded domain with smooth boundary $\partial\Omega = \Sigma$, and let $\nu(\cdot)$ denote the outward unit normal vector to $\Sigma$. The Master Equation takes the following form, 
\begin{equation}\label{master_equation_bounded}
    \begin{cases}
        \begin{aligned}
        &-\partial_t U(t,x,m) - \Delta_x U(t,x,m) + H(x, D_x U(t,x,m)) \\
        &\quad - \int_{\Omega} \nabla_y \cdot \big[D_m U\big](t,x,m,y) \, dm(y) \\
        &\quad + \int_{\Omega} D_m U(t,x,m,y) \cdot H\big(y, D_x U(t,y,m)\big) \, dm(y) = \mathscr{F}(x,m), \\
        &\qquad \text{in } (0,T) \times \Omega \times \mathcal{P}(\Omega), \\[1ex]
        
        &D_x U(t,x,m) \cdot \nu(x) = 0, \\
        &\qquad \text{on } (0,T) \times \Sigma \times \mathcal{P}(\Omega), \\[1ex]
        
        &D_m U(t,x,m,y) \cdot \nu(y) = 0, \\
        &\qquad \text{on } (0,T) \times \Omega \times \mathcal{P}(\Omega) \times \Sigma, \\[1ex]
        
        &U(T,x,m) = \mathscr{G}(x,m), \\
        &\qquad \text{in } \Omega \times \mathcal{P}(\Omega).
        \end{aligned}
    \end{cases}
\end{equation}

The well-posedness of this equation has been established in \cite{linear_in_Om}. In what follows, we present the key results that will be used later. While these results are analogous to the periodic case, careful attention must be paid to the boundary conditions. 
In this section, we impose the following assumptions.  
\begin{enumerate}
    \item \textbf{Neumann boundary conditions:}
    \begin{equation}
        \begin{aligned}
            &\left\langle \frac{\delta \mathscr{F}}{\delta m}(x,m,y), \nu(y) \right\rangle \Big|_{\Sigma} = 0, \\
            &\left\langle \frac{\delta \mathscr{G}}{\delta m}(x,m,y), \nu(y) \right\rangle \Big|_{\Sigma} = 0, \\
            &\left\langle D_x \mathscr{G}(x,m), \nu(x) \right\rangle \Big|_{\Sigma} = 0,
        \end{aligned}
    \end{equation}
    for all $m \in \mathcal{P}(\Omega)$.
    
    \item \textbf{Hamiltonian regularity:} The Hamiltonian $H: \mathbb{T}^n \times \mathbb{R}^n \to \mathbb{R}$ is smooth and globally Lipschitz continuous, and satisfies the uniform ellipticity condition, 
    \begin{equation}
        0 < D^2_{pp} H(x,p) \leq C I_n, \quad \text{for all } (x,p) \in \mathbb{T}^n \times \mathbb{R}^n,
    \end{equation}
    where $C > 0$ is a constant and $I_n$ denotes the $n \times n$ identity matrix.
    
    \item \textbf{Monotonicity conditions:} The operators $\mathscr{F}, \mathscr{G}: \mathbb{T}^n \times \mathcal{P}(\Omega) \to \mathbb{R}$ are globally Lipschitz continuous and satisfy the following monotonicity conditions for any $m, m' \in \mathcal{P}(\Omega)$, 
    \begin{equation}
        \begin{aligned}
            &\int_{\Omega} \left(\mathscr{F}(x,m) - \mathscr{F}(x,m')\right) d(m - m')(x) \geq 0, \\
            &\int_{\Omega} \left(\mathscr{G}(x,m) - \mathscr{G}(x,m')\right) d(m - m')(x) \geq 0.
        \end{aligned}
    \end{equation}
\end{enumerate}
Additionally, we assume that all regularity conditions mentioned in Section~\ref{IPMFGMaster} continue to hold in this setting. Again,  we present the key results from  \cite{linear_in_Om}, reformulated for our application.
\begin{theorem}[Theorem 2.5 in \cite{linear_in_Om}]\label{well-pose-bounded}
    Under all conditions listed above, the Master Equation \eqref{master_equation_bounded} admits a unique classical solution.
\end{theorem}

The corresponding mean field game (MFG) system with Neumann boundary conditions is given by 
\begin{equation}\label{mfg-bounded}
    \begin{cases}
        -\partial_t u(t,x) - \Delta u(t,x) + H(x, \nabla u) = \mathscr{F}(x, m(t,x)), & \text{in } (t_0,T)\times\Omega, \\
        \partial_t m(t,x) - \Delta m(t,x) - \operatorname{div}\left(m(t,x) H_p(x, \nabla u)\right) = 0, & \text{in } (t_0,T)\times\Omega, \\
        \partial_\nu u(t,x) = 0, \quad \left[\nabla m(t,x) + m H_p(x, \nabla u)\right] \cdot \nu(x) = 0, & \text{on } (0,T) \times \Sigma, \\
        u(T,x) = \mathscr{G}(x, m_T), & \text{in } \Omega, \\
        m(t_0, x) = m_0(x), & \text{in } \Omega.
    \end{cases}  		
\end{equation}

Given a solution $(u,m)$ of the system \eqref{mfg-bounded}, we define a function $U$ by $U(t_0,x,m_0) \equiv u(t_0,x)$. Then $U(t,x,m)$ solves the Master Equation.

We also need to study the linearized system of the Master Equation. As established in \cite{linear_in_Om}, we have the following result. 

\begin{theorem}[Section 5 in \cite{linear_in_Om}]\label{key-domain}
    Let $(u,m)$ be a solution of the system \eqref{mfg-bounded}. Consider the linearized system, 
    \begin{equation}\label{linearized system 2}
        \begin{cases}
            -\partial_t v - \Delta v + H_p(x, \nabla u) \cdot \nabla v = \dfrac{\delta \mathscr{F}}{\delta m}(x, m(t))(\rho(t)), & \text{in } (t_0,T) \times \Omega, \\
            \partial_t \rho - \Delta \rho - \nabla \cdot \left(m D^2_{pp} H(x, \nabla u) \nabla v\right) - \nabla \cdot \left(\rho H_p(x, \nabla u)\right) = 0, & \text{in } (t_0,T) \times \Omega, \\
            \partial_\nu v(t,x) = 0, \quad \left[\nabla \rho + \rho H_p(x, \nabla u)\right] \cdot \nu(x) = 0, & \text{on } (t_0,T) \times \Sigma, \\
            v(T,x) = \dfrac{\delta \mathscr{G}}{\delta m}(x, m(T))(\rho(T)), & \text{in } \Omega, \\
            \rho(t_0,x) = \mu_0(x), & \text{in } \Omega.
        \end{cases}
    \end{equation}

    Let $\mu_0 \in C^{\infty}(\Omega)$. If all conditions in Theorem~\ref{well-pose-bounded} are satisfied, then the linearized system has a solution and
    \begin{equation}
        v(t_0,x) = \int_{\Omega} \frac{\delta U}{\delta m}(t_0,x,m_0,y) \mu_0(y) \, dy,
    \end{equation}
    where $U$ is the solution of the Master Equation \eqref{master_equation_bounded} with the initial condition $m(t_0,x) = m_0(x)$.
\end{theorem}
In fact, the results established in Section~\ref{IPMFGMaster} extend to the bounded domain $\Omega$ via analogous arguments.
\begin{corollary}
    Under the assumptions for the bounded domain $\Omega$, Theorems~\ref{Main2} and~\ref{Main3} remain valid for Master Equations defined on $\mathcal{P}(\Omega)$.
\end{corollary}

\subsection{Inverse boundary problems for the Master Equation}

In the previous section, we used time boundary data for recovery purposes. We now use spatial boundary data to identify the running cost. Define the measurement map by 
\begin{equation}
    \mathcal{L}_{\mathscr{F}}(\mathscr{G}) \equiv U(t,x,m)\big|_{\Sigma}.
\end{equation}
We now state the main result of this section.

\begin{theorem}\label{domain-main}
    Consider two Master Equations defined on $\Omega$ with running costs $\mathscr{F}_j(x,m)$ for $j=1,2$. Assume that 
    \begin{itemize}
        \item[(1)] The Hamiltonian is given by $H(x,p) = \frac{1}{2}|p|^2$;
        \item[(2)] The linearized running costs vanish at the static solutions, i.e.,
              \begin{equation}
                  \frac{\delta \mathscr{F}_j}{\delta m}(x, m_j^{(0)}(x)) = 0 \quad \text{for } j=1,2,
              \end{equation}
              where $m_j^{(0)}$ is a static solution of the corresponding MFG system;
        \item[(3)] The measurement maps coincide: $\mathcal{L}_{\mathscr{F}_1}(\mathscr{G}) = \mathcal{L}_{\mathscr{F}_2}(\mathscr{G})$ for all admissible $\mathscr{G}$.
    \end{itemize}
    Then
\begin{equation*}
    (\nabla u_1^{(0)}, m_1^{(0)}) = (\nabla u_2^{(0)}, m_2^{(0)}), 
\end{equation*}
where $(u_j^{(0)}, m_j^{(0)})$ ($j=1,2$) are static solutions for the corresponding MFG system associated with \eqref{mfg-bounded}. 

If we further assume that $\mathscr{F}_i(x,m)$ for $i=1,2$ depend on $m$ locally and admit the power series expansion
        \begin{equation}
            \mathscr{F}_i(x,m) = \sum_{k=2}^{\infty} \frac{1}{k!} \mathscr{F}_i^{(k)}(x) (m - m_i^{(0)})^k,
        \end{equation}
then       
    \begin{equation}
        \mathscr{F}_1(x,m) = \mathscr{F}_2(x,m) \quad \text{in } \Omega \times \mathcal{P}(\Omega).
    \end{equation}
\end{theorem}
\subsection{Proof of Theorem~\ref{domain-main}}

This subsection is devoted to the proof of Theorem~\ref{domain-main}. We begin by establishing some auxiliary results.

\begin{lemma}\label{D-data}
    Let $\mathcal{L}_{\mathscr{F}_i}$ for $i=1,2$ be the measurement maps associated with the Master Equation with running costs $\mathscr{F}_i(x,m)$. Fix $t_0 \in (0,T)$ and $m_0 \in \mathcal{P}(\Omega)$, and let $(u_i, m_i)$ be solutions of the corresponding MFG systems with the initial condition $m_i(t_0,\cdot) = m_0$ for $i=1,2$. If $\mathcal{L}_{\mathscr{F}_1} = \mathcal{L}_{\mathscr{F}_2}$, then
    \begin{equation}
        u_1(t,x) = u_2(t,x) \quad \text{on } [t_0,T) \times \Sigma.
    \end{equation}
\end{lemma}

\begin{proof}
    The equality $\mathcal{L}_{\mathscr{F}_1} = \mathcal{L}_{\mathscr{F}_2}$ implies by definition that $u_1(t_0,x) = u_2(t_0,x)$ on $\Sigma$. Now fix an arbitrary $\tau \in (t_0,T)$ and let $U_i$ be the solution of the Master Equation with running cost $\mathscr{F}_i(x,m)$ for $i=1,2$. Then
    \begin{equation}
        U_1(\tau,x, m_1(\tau)) = U_2(\tau,x, m_2(\tau)) \quad \text{on } \Sigma.
    \end{equation}
    By the uniqueness of solutions to the MFG system, we conclude that
    \begin{equation}
        u_1(\tau,x) = u_2(\tau,x) \quad \text{on } \Sigma.
    \end{equation}
    Since $\tau \in (t_0,T)$ is arbitrary, the identity holds for all $t \in [t_0,T)$.
\end{proof}

\begin{lemma}\label{D-data-key}
    Let $f_j \in C^{2+\alpha,N}(\Omega)$ for $j=1,2$, where the superscript $N$ indicates that $f_j$ satisfies the Neumann boundary condition. Let $u_i(t,x)$ be the solution to the following systems, 
    \begin{equation}
        \begin{cases}        
            \partial_t u_i - \Delta u_i - \nabla \cdot (u_i \nabla f_i) = 0, & \text{ in } (0,T) \times \Omega, \\
            \partial_{\nu} u_i(t,x) = 0, & \text{ on } (0,T) \times \Sigma, \\
            u_i(0,x) = g(x), & \text{ in } \Omega,  
        \end{cases}
    \end{equation}
    for $i=1,2$. Suppose that $u_1(T,x) = u_2(T,x)$ for all $g \in C^{\alpha}(\Omega)$. Then 
    \begin{equation}
        \nabla f_1 = \nabla f_2 \quad \text{in } \Omega.
    \end{equation}
\end{lemma}

\begin{proof}
    Define the transformation $v_i = u_i e^{f_i/2}$ for $i=1,2$. A direct computation shows that $v_i$ satisfies 
    \begin{equation}
        \partial_t v_i = \Delta v_i - V_i v_i, \quad \text{} V_i \equiv \frac{1}{4}|\nabla f_i|^2 - \frac{1}{2} \Delta f_i,
    \end{equation}
    where the initial condition $v_i(0,x) = e^{f_i(x)/2} g(x)$ for $i=1,2$. The Neumann boundary condition $\partial_{\nu} v_i = 0$ follows from $\partial_{\nu} u_i = 0$ by the boundary condition and $\partial_{\nu} f_i = 0$ by the assumption.

    Define the self-adjoint operators $H_i = -\Delta + V_i$ with domain incorporating the zero Neumann conditions. The solution for $v_i$ is given by 
    \begin{equation*}
        v_i(t) = \exp(-t H_i)(e^{f_i/2} g),
    \end{equation*}
    and hence
    \begin{equation*}
        u_i(t,x) = e^{-f_i/2} v_i(t,x) = e^{-f_i/2} \exp(-t H_i)(e^{f_i/2} g).
    \end{equation*}

    Define the multiplication operators 
    \begin{equation*}
        A_i : \phi \mapsto e^{f_i/2} \phi, \quad A_i^{-1} : \phi \mapsto e^{-f_i/2} \phi.
    \end{equation*}
    Then the solution operator for the original equation is 
    \begin{equation*}
        S_i(t) = A_i^{-1} \exp(-t H_i) A_i.
    \end{equation*}

    The condition $u_1(T,x) = u_2(T,x)$ for all $g$ implies that $S_1(T) = S_2(T)$, i.e.,
    \begin{equation*}
        A_1^{-1} \exp(-T H_1) A_1 = A_2^{-1} \exp(-T H_2) A_2.
    \end{equation*}
    Multiplying both sides on the left by $A_1$ and on the right by $A_2^{-1}$ yields 
    \begin{equation}\label{trans to operator}
        \exp(-T H_1) (A_1 A_2^{-1}) = (A_1 A_2^{-1}) \exp(-T H_2).
    \end{equation}

    Let $\varphi(x) = e^{(f_1(x) - f_2(x))/2}$, which is the multiplication operator $A_1 A_2^{-1}$. Then equation \eqref{trans to operator} becomes 
    \begin{equation*}
        \exp(-T H_1) \varphi = \varphi \exp(-T H_2).
    \end{equation*}

    Let $\{\lambda_k\}_{k=1}^{\infty}$ be the eigenvalues of $H_2$ with corresponding orthonormal eigenfunctions $\{\psi_k\}$. For each $k$, 
    \begin{equation}\label{lemma-eigen}
        \exp(-T H_1)(\varphi \psi_k) = \varphi \exp(-T H_2) \psi_k = e^{-T \lambda_k} \varphi \psi_k.
    \end{equation}
    Thus, ~\eqref{lemma-eigen} implies that $\varphi \psi_k$ is an eigenfunction of $\exp(-T H_1)$ with eigenvalue $e^{-T \lambda_k}$. Since the exponential map is injective and $H_1$ is self-adjoint, there exists an eigenvalue $\mu$ of $H_1$ such that $\mu = \lambda_k$; accordingly,  $\varphi \psi_k$ is an eigenfunction of $H_1$ with eigenvalue $\lambda_k$. Therefore, we have 
    \begin{equation*}
        H_1(\varphi \psi_k) = \lambda_k \varphi \psi_k = \varphi (H_2 \psi_k),
    \end{equation*}
    which implies that 
    \begin{equation}
        (H_1 \varphi - \varphi H_2) \psi_k = 0 \quad \text{for all } k.
    \end{equation}
    Since $\{\psi_k\}$ forms a basis of $L^2(\Omega)$, we conclude  $H_1 \varphi = \varphi H_2$ as operators.
    Recall that $H_i = -\Delta + V_i$. For any smooth test function $h$, we have 
     $ H_1(\varphi h) = \varphi (H_2 h)$, yielding  
    \begin{equation}
        -\Delta(\varphi h) + V_1 \varphi h = -\varphi \Delta h + \varphi V_2 h.
    \end{equation}
    Expanding the Laplacian term, we have 
    \begin{equation}
        -h \Delta \varphi - 2 \nabla \varphi \cdot \nabla h - \varphi \Delta h + V_1 \varphi h = -\varphi \Delta h + \varphi V_2 h.
    \end{equation}
    Simplifying the above equation leads to 
    \begin{equation}\label{final equation in lemma domain}
        -2 \nabla \varphi \cdot \nabla h - h \Delta \varphi + V_1 \varphi h = \varphi V_2 h.
    \end{equation}

    Now fix an arbitrary point $x_0 \in \Omega$ and choose $h$ such that $h(x_0) = 0$ and $\nabla h(x_0) \neq 0$. Evaluating \eqref{final equation in lemma domain} at $x_0$ gives 
    \begin{equation}
        -2 \nabla \varphi(x_0) \cdot \nabla h(x_0) = 0.
    \end{equation}
    Since $\nabla h(x_0)$ is arbitrary, it follows that $\nabla \varphi(x_0) = 0$. As $x_0$ is arbitrary, $\nabla \varphi = 0$ throughout $\Omega$, implying that  $\varphi$ is constant.

    Substituting $\nabla \varphi = 0$ and $\Delta \varphi = 0$ into \eqref{final equation in lemma domain} yields 
    \begin{equation}
        V_1 \varphi h = \varphi V_2 h \quad \Rightarrow \quad V_1 = V_2.
    \end{equation}
    Since $\varphi = e^{(f_1 - f_2)/2}$ is constant, $f_1 - f_2$ is constant, and hence $\nabla f_1 = \nabla f_2$.
\end{proof}

We now present the proof of the main result in this section.

\begin{proof}[Proof of Theorem~\ref{domain-main}]
    We first establish that $(\nabla u_1^{(0)}, m_1^{(0)}) = (\nabla u_2^{(0)}, m_2^{(0)})$.

    Consider the linearized system around the static solution for $j=1,2$, 
    \begin{equation}\label{linearized system-1/2}
        \begin{cases}
            -\partial_t u_j^{(1)} - \Delta u_j^{(1)} + \nabla u_j^{(0)} \cdot \nabla u_j^{(1)} = 0, & \text{in } (t_0,T) \times \Omega, \\
            \partial_t m_j^{(1)} - \Delta m_j^{(1)} - \nabla \cdot (m_j^{(0)} \nabla u_j^{(1)}) - \nabla \cdot (m_j^{(1)} \nabla u_j^{(0)}) = 0, & \text{in } (t_0,T) \times \Omega, \\
            \partial_\nu m_j^{(1)}(t,x) = \partial_\nu u_j^{(1)}(t,x) = 0, & \text{on } (t_0,T) \times \Sigma, \\
            u_j^{(1)}(T,x) = \dfrac{\delta \mathscr{G}}{\delta m}(x, m_j^{(0)})(m_j^{(1)}(T)), & \text{in } \Omega, \\
            m_j^{(1)}(t_0,x) = \mu_0, & \text{in } \Omega.
        \end{cases}
    \end{equation}

    Since $\mathcal{L}_{\mathscr{F}_1}(\mathscr{G}) = \mathcal{L}_{\mathscr{F}_2}(\mathscr{G})$, Lemma~\ref{D-data} implies that 
    \begin{equation}
        u_1(t,x) = u_2(t,x) \quad \text{on } \Sigma.
    \end{equation}
    Theorem~\ref{key-domain} then yields 
    \begin{equation}
        u_1^{(1)}(t,x) = u_2^{(1)}(t,x) \quad \text{on } \Sigma.
    \end{equation}

    Let $\rho(t,x)$ be the solution to the adjoint system, 
    \begin{equation}
        \begin{cases}
            \partial_t \rho - \Delta \rho - \nabla \cdot (\rho \nabla u_2^{(0)}) = 0, & \text{ in } (t_0,T) \times \Omega, \\
            \partial_\nu \rho(t,x) = 0, & \text{ on } (t_0,T) \times \Sigma, \\
            \rho(t_0,x) = \mu_0, & \text{ in } \Omega.
        \end{cases}
    \end{equation}

    To construct appropriate terminal conditions, we employ a sequence of smooth approximations. Let $\{\phi_n\}_{n=1}^\infty \subset C^\infty(\Omega)$ be a sequence of smooth functions satisfying the following conditions:  
    \begin{itemize}
        \item[(1)] $\phi_n \geq 0$ for all $n$;
        \item[(2)] $\int_\Omega \phi_n(x)  dx = 1$ for all $n$;
        \item[(3)] $\text{supp}(\phi_n) \subset B_{1/n}(y_0)$ for some fixed $y_0 \in \Omega$;
        \item[(4)] $\phi_n \to \delta_{y_0}$ in the sense of distributions as $n \to \infty$.
    \end{itemize}

    For each $n$, we choose $\mathscr{G}_n(x,m_T)=\int_{\Omega}\phi_n(y)m(T,y)\,dy$, so that we have 
    \[
    \dfrac{\delta \mathscr{G}_n}{\delta m}(x, m_j^{(0)})\left(m_j^{(1)}(T)\right)=\int_{\Omega} \phi_n(y)m_j^{(1)}(T,y)\,dy
    \]
    for $j=1,2$. One can verify that for each $n$, the pair $(u_{2,n}^{(1)}, m_{2,n}^{(1)}) = \left(\int_{\Omega} \phi_n(x) \rho(T,x)  dx, \rho(t,x)\right)$ solves the linearized system~\eqref{linearized system-1/2} for $j=2$.

    Define $\overline{u}_n = u_{1,n}^{(1)} - u_{2,n}^{(1)}$. Then $\overline{u}_n$ satisfies 
    \begin{equation}\label{1-2 domain n}
        \begin{cases}
            -\partial_t \overline{u}_n - \Delta \overline{u}_n + \nabla u_1^{(0)} \cdot \nabla \overline{u}_n = 0, & \text{in } (t_0,T) \times \Omega, \\
            \partial_\nu \overline{u}_n = 0, \quad \overline{u}_n = 0, & \text{on } (t_0,T) \times \Sigma, \\
            \overline{u}_n(T,x) = \phi_n(x)\left(m_{1,n}^{(1)}(T) - m_{2,n}^{(1)}(T)\right), & \text{in } \Omega.
        \end{cases}
    \end{equation}

    By the unique continuation principle (see \cite{saut1987unique}), we obtain $\overline{u}_n = 0$ and hence $u_{1,n}^{(1)}(t,x) = u_{2,n}^{(1)}(t,x) = \int_{\Omega} \phi_n(x) \rho(T,x)\,  dx$ for each $n$. Consequently, $m_{1,n}^{(1)}(t,x)$ satisfies 
    \begin{equation}\label{equation for m n}
        \begin{cases}
            \partial_t \rho - \Delta \rho - \nabla \cdot (\rho \nabla u_1^{(0)}) = 0, & \text{ in } (t_0,T) \times \Omega, \\
            \partial_\nu \rho(t,x) = 0, & \text{ on } (t_0,T) \times \Sigma, \\
            \rho(t_0,x) = \mu_0, & \text{ in } \Omega,
        \end{cases}
    \end{equation}
    and
    \begin{equation}
        \int_{\Omega} \phi_n(x) m_{1,n}^{(1)}(T,x)\,  dx = \int_{\Omega} \phi_n(x) m_{2,n}^{(1)}(T,x)  \,dx, \quad \text{for all } n \in \mathbb{N}.
        \label{deltalimit}
    \end{equation}

    Since $ m_{1,n}^{(1)}(T,x)$ and $m_{2,n}^{(1)}(T,x)$ are independent of $n$, taking the limit as $n \to \infty$ in \eqref{deltalimit}, we obtain $m_1^{(1)}(T, y_0) = m_2^{(1)}(T, y_0)$ for all $y_0 \in \Omega$.

    Applying Lemma~\ref{D-data-key}, we conclude that $\nabla u_1^{(0)} = \nabla u_2^{(0)}$. That is, both systems share the same static solution.

 Since we have assumed that the linearized running costs vanish at the static solutions (the assumption (2) of Theorem~\ref{domain-main}), namely, $\frac{\delta \mathscr{F}_j}{\delta m}(x, m_j^{(0)}(x)) = 0$ for $j=1,2$, and we have also established the equality $(\nabla u_1^{(0)}, m_1^{(0)}) = (\nabla u_2^{(0)}, m_2^{(0)})$, applying Theorem 2.9 in \cite{Cauchydata-zs} and the unique continuation principle for mean field games in \cite[Section 6]{Cauchydata-zs} leads to 
\begin{equation*}
    \mathscr{F}_1(x,m) = \mathscr{F}_2(x,m).
\end{equation*}
\end{proof}

\section*{Acknowledgment} 
 The work of H. Liu was supported by the Hong Kong RGC General Research Funds (projects 11311122, 11300821, and 11303125), the NSFC/RGC Joint Research Fund (project  N\_CityU101/21), the France-Hong Kong ANR/RGC Joint Research Grant, A-CityU203/19. 
The work of J. Qian was partially supported by NSF grants 2152011 and 2309534 and an MSU SPG grant.

\bibliographystyle{plain}
\bibliography{ref}

\end{document}